\documentclass{amsart}

\usepackage{amsmath, amsthm}%
\usepackage{amsfonts}%
\usepackage{amssymb}%
\usepackage{graphicx}
\usepackage{amscd, amstext, hyperref}
\usepackage{pictexwd}
\usepackage{dcpic}
\usepackage[mathscr]{eucal}
\usepackage{multicol}
\usepackage{multirow}
\usepackage{color}
\usepackage{bm}
\usepackage[margin=1in]{geometry}
\usepackage[dvipsnames]{xcolor}
\usepackage[section]{placeins} 
\usepackage{float}
\newtheorem{theorem}{Theorem}[section]
\newtheorem{lemma}[theorem]{Lemma}
\newtheorem{proposition}[theorem]{Proposition}
\newtheorem{corollary}[theorem]{Corollary}
\newtheorem{algorithm}[theorem]{Algorithm}

\newtheorem{problem}[theorem]{Problem}
\newtheorem*{maintheorem*}{Main Theorem}
\newtheorem*{corollary*}{Corollary}
\newtheorem*{lemma*}{Lemma}
\newtheorem*{keylemma*}{Key Lemma}

\theoremstyle{definition}
\newtheorem{definition}[theorem]{Definition}

\newtheorem{example}[theorem]{Example}

\theoremstyle{remark}
\newtheorem{remark}[theorem]{Remark}

\numberwithin{equation}{section}
\numberwithin{table}{section}

\newcommand{\C}{\mathbb{C}}

\DeclareMathOperator{\Jac}{Jac}

\newcommand{\ee}[2]{{}e_{#2}^{#1}}

\newcommand{\aef}[1]{{\color{blue} \sf AEF: [#1]}}

\usepackage{pgfplots}
\usepackage{mathrsfs}
\usetikzlibrary{automata,arrows,positioning,calc}
\usepackage{tikz}
\usepackage{tikz-cd}
\usetikzlibrary {arrows.meta}
\usetikzlibrary{cd}

\begin{document}

\title{The five-color hypercube Adinkra and the Jacobian of a generalized Fermat curve}
\author{Amanda E. Francis and Ursula A. Whitcher}

\maketitle

\begin{abstract}
Adinkras are highly structured graphs developed to study 1-dimensional supersymmetry algebras. A cyclic ordering of the edge colors of an Adinkra, or rainbow, determines a Riemann surface and a height function on the vertices of the Adinkra determines a divisor on this surface. We study the induced map from height functions to divisors on the Jacobian of the Riemann surface. In the first nontrivial case, a 5-dimensional hypercube corresponding to a Jacobian given by a product of 5 elliptic curves each with $j$-invariant 2048, we develop and characterize a purely combinatorial algorithm to compute height function images. We show that when restricted to a single elliptic curve, every height function is a multiple of a specified generating divisor, and raising and lowering vertices corresponds to adding or subtracting this generator. We also give strict bounds on the coefficients of this generator that appear in the collection of all divisors of height functions.
\end{abstract}

\tableofcontents

\section{Introduction}\label{S:intro}

Adinkras are highly structured graphs whose combinatorial structure encapsulates the data of 1-dimensional supersymmetry algebras. Initially described by the physicists M. Faux and S.J. Gates, Jr. (see \cite{FG}), Adinkras have also drawn attention for their intriguing combinatorial properties and their relationship to arithmetic geometry. In the present work, we examine the connection between the \emph{height function} of an Adinkra and a divisor on the Jacobian of an associated Riemann surface. We may view the construction as an algebraic thread that aids us in navigating a combinatorial labyrinth of potential height functions. Alternatively, we may see ourselves as number-theoretic adventurers, using the combinatorial data to examine the arithmetic properties of an intriguing high-genus curve.

Let us describe our motivation and results in more detail. Recall that an Adinkra consists of an $N$-regular bipartite graph equipped with an $N$-coloring of its edges and satisfying certain properties, called a \emph{chromotopology}, together with two additional structures, the aforementioned height function and an \emph{odd dashing}. (For a more detailed discussion, see \S\ref{SS:AdinkraDef}.) The authors of \cite{codes} obtained a fundamental classification result, showing that every $N$-colored Adinkra chromotopology can be obtained as a quotient of the hypercube $H^N$ by a \emph{doubly even code}, where $H^N$ denotes the hypercube with $2^N$ vertices in $\{0,1\}^N$ and vertices $v$ and $v'$ are adjacent if $v$ and $v'$ differ in exactly one coordinate. Thus, understanding the possible height functions and odd dashings for  hypercube Adinkras is of key importance for mapping Adinkra geography.

The number of possible height functions for hypercube Adinkras grows rapidly. Though for $N=1$ there are only two possible heights for $H^1$, depending on whether a black or white vertex is placed on top, for $N=2,3,4,5$ there are $6, 38, 990$, and $395094$ possible height functions, respectively, as computed in \cite{Zhang}. Finding methods to characterize or classify height functions is thus of great interest. We count the $N=6$ hypercube Adinkra height functions in Section~\ref{SS:CountingHeights}, using an observation of Steven Charlton, before turning our attention back to the structure of $N=5$ hypercube height functions. The enumeration and structure of heights for quotient Adinkras is still to a great extent an open problem. 

The authors of \cite{geom1} constructed a Riemann surface associated to a \emph{rainbow}, or cyclic order on edge colors. In the subsequent work \cite{geom2}, they described a map from height functions to divisors on the Jacobian of the Riemann surface. In the case of hypercube Adinkras $H^N$, for $N=1, 2$ and $3$, the group of divisors on the appropriate Jacobian is trivial; for $N=4$, though the group is nontrivial, the image of the map is the identity element. For $N=5$, the authors of \cite{geom2} showed that the image of the map from height functions to Jacobian divisors contains at least one non-identity divisor. As there are $395094$ distinct height functions, understanding the map in more detail is of pressing concern.

We characterize the map from height functions on the $N=5$ hypercube Adinkra $H^5$ to Jacobian divisors in three ways. Geometrically, we give an explicit map from the Riemann surface $X_\mathrm{alg}^5$ to a product of elliptic curves $E_1 \times \cdots \times E_5$, which form an isogeneous decomposition of ${\rm Jac}(X_{\rm alg}^5)$. The elliptic curves have $j$-invariant 2048, and thus are isogenous over $\mathbb{C}$. Working arithmetically, for each curve $E_k$ we identify a number field $K_k$ where the height function image is defined (see Proposition~\ref{prop:field_choice}) and characterize the subgroup of Mordell-Weil generated by the image. 

We show that the height function map can be computed via a purely combinatorial algorithm. This algorithm uses the notion of \emph{$j$-color splittings} and  \emph{total-color-splittings}, which use the rainbow to divide each of the sets of black and white vertices into subsets labeled $+$ and $-$ (see Definitions \ref{def:jcs} and \ref{def:cs}). We first demonstrate that the images of vertices and face centers of $H^5$ in the Mordell-Weil group $MW(E_k)$ over $K_k$ lie in a group isomorphic to $\mathbb{Z} \times \mathbb{Z}/(4) \times \mathbb{Z}/(2)$; we give a complete description of the generators of this group and the images of the points.

Using the relationship between height functions and divisors, we obtain our Main Theorem: 

\begin{maintheorem*}\label{thm:main}
Let $\nu \colon X_{\rm alg}^5 \to E_1 \times \cdots \times E_5$  be the map that takes a height function on $H^5$ to the product of the elliptic curves $E_k$ (the isogeneous decomposition of the Jacobian of $X_{\rm alg}^5$), and let $\nu_k(h)$ be the projection of $\nu$ to $E_k$. Let $h_1$ be the height function on $H^5$ obtained by lowering the top vertex of the fully extended height function, and let $h$ be any height function on $H^5$. Then 
$\nu_k(h)  = a \nu_k(h_1)$ for some integer $a$, where $-8 \leq a \leq 8$.
Furthermore, if $h$ and $h'$ are two height functions on $H^5$ that differ by lowering (or raising) a single vertex, we may write $\nu_k(h) = a \nu_k(h_1)$ and $\nu_k(h') = a'\nu_k(h_1)$ with $|a-a'|=1$.
\end{maintheorem*}

Note that lowering a vertex in a particular height may correspond to adding the generator in one elliptic curve but subtracting the generator in a different elliptic curve, so the choice of $a$ in the above theorem depends on both $k$ and $h$. The details of the sign choice corresponding to lowering a vertex are determined by the total-color-splitting associated to the rainbow. 

The plan of the paper is as follows. We review the combinatorics of Adinkras and define color splitting in \S~\ref{S:Combinatorics}. In \S~\ref{S:Geometrization}, we examine the Riemann surface associated to the $N=5$ hypercube Adinkra $H^5$ together with its Jacobian and characterize the images of the Adinkra's vertices and face centers in the elliptic curve components of the Jacobian. We summarize this description in Theorem~\ref{thm:imagesinEC}. We conclude this section by discussing appropriate fields of definition. In \S~\ref{S:Heights_Divisors_5}, we use Theorem~\ref{thm:imagesinEC} to explore the relationship between height functions and the associated divisors and prove our \hyperref[thm:main]{Main Theorem}.

Code and data associated with this paper may be found in the Zenodo repository \cite{ourcode}.

\section{Combinatorics}\label{S:Combinatorics}

\subsection{Adinkras and chromotopologies}\label{SS:AdinkraDef}

An {\it Adinkra topology} is an $N$-regular bipartite graph $A$. We call the two vertex sets in the bipartition the black and white vertices, respectively. 
 A {\it chromotopology} is an Adinkra topology $A$ equipped with an $n$-coloring of the edges of $A$, such that for any two distinct colors $j$ and $k$, $\{e \in E(A) \mid {\rm color}(e) \in \{j,k\}\}$ is comprised of disjoint  4-cycles. As we note in \S~\ref{S:intro}, the authors of \cite{codes} showed that any chromotopology can be obtained as a quotient of a hypercube by a doubly even code.
 
 A {\it height function} on a chromotopology $A$ is a map $h: V(A) \to \mathbb Z_{\geq 0}$, such that $|h(v_1) - h(v_2)| = 1 $ if $(v_1,v_2) \in E(A)$. We shall say that two heights $h$ and $h'$ are the same if they differ by a constant (i.e., there exists $a \in \mathbb{Z}$ such that $h(v) = h'(v)+a$ for all $v \in V(A)$). Our convention will be to use the height representative that satisfies $\min_{v \in V(A)} h(v) = 0$. 
A {\it ranked chromotopology} is a chromotopology $A$ equipped with such a {\it height function}. 

A  {\it dashing function} is a map $d: E(A) \to \mathbb Z/(2)$, such that $\sum_{e \in C} d(e) = 1 \in \mathbb Z/(2)$ for any 4-cycle in $A$. That is, under $d$, some edges of $A$ are dashed and some are not, and there must be an odd number of dashings in any 4-cycle in the graph.
A {\it dashed} chromotopology is a chromotopology $A$ equipped with such a dashing function.  

 An {\it Adinkra} is a dashed, ranked chromotopology. In Figure \ref{fig:small_adinkra} an $N=3$ Adinkra is shown.  
 This is an example of a {\it fully extended} Adinkra; that is, $\max\{|h(v) - h(v')| : v,v' \in V(G)\}$ is maximal among all height functions for the given chromotopology. A fully extended $N=5$ Adinkra is shown on the right in Figure \ref{fig:basic_heights}. The image on the left is called a {\it valise} Adinkra for the chromotopology; that is, ${\max\{|h(v) - h(v')| : v,v' \in V(G)\}}$ is minimal among all height functions for the given chromotopology.

\begin{figure}[ht!]
\begin{tikzpicture}
    \draw[color = blue, line width = .8pt, dashed] (0,3) -- (1,2);
    \draw[color = blue, line width = .8pt] (0,1) -- (-1,2);
    \draw[color = blue, line width = .8pt] (1,1) -- (0,2);
    \draw[color = blue, line width = .8pt] (0,0) -- (-1,1);
    \draw[color = Green, line width = .8pt, dashed] (0,3) -- (-1,2);
    \draw[color = Green, line width = .8pt, dashed] (0,2) -- (-1,1);
    \draw[color = Green, line width = .8pt,dashed] (1,2) -- (0,1);
    \draw[color = Green, line width = .8pt] (1,1) -- (0,0);
    \draw[color = orange, line width = .9pt, dashed] (-1,1) -- (-1,2);
    \draw[color = orange, line width = .9pt] (0,3) -- (0,2);
    \draw[color = orange, line width = .9pt] (0,1) -- (0,0);
    \draw[color = orange, line width = .9pt] (1,2) -- (1,1);
    \draw[color = gray, line width = .3pt,dotted] (1.5,3)--(-1.5,3);
    \draw[color = gray, line width = .3pt,dotted] (1.5,2)--(-1.5,2);
    \draw[color = gray, line width = .3pt,dotted] (1.5,1)--(-1.5,1);
    \draw[color = gray, line width = .3pt,dotted] (1.5,0)--(-1.5,0);
    
    \draw [fill=white] (0,3) circle (2pt);
    \draw [fill=black] (1,2) circle (2pt);
    \draw [fill=black] (0,2) circle (2pt);
    \draw [fill=black] (-1,2) circle (2pt);
    \draw [fill=white] (-1,1) circle (2pt);
    \draw [fill=white] (0,1) circle (2pt);
    \draw [fill=white] (1,1) circle (2pt);
    \draw [fill=black] (0,0) circle (2pt);
    \node[color = gray] at (2,3) {$h=3$};
    \node[color = gray] at (2,2) {$h=2$};
    \node[color = gray] at (2,1) {$h=1$};
    \node[color = gray] at (2,0) {$h=0$};
\end{tikzpicture}
\caption{An $N=3$ Adinkra $A$. The dashing function $d$ on $A$ satisfies $d(e)=1$ if and only if the edge is dashed in the visualization above. The height function values correspond to the $y$-coordinates of the vertices.}\label{fig:small_adinkra}
\end{figure}
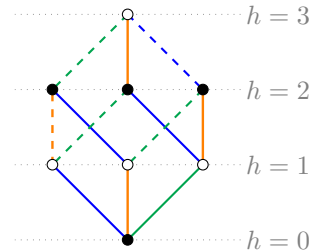

\begin{remark}\label{rem:pegs}
An equivalent way to define an Adinkral height, which requires less data in some sense, is the following. It is natural to think of the Adinkras as `hanging gardens' in which certain vertices (which we will call `pinned') are secured at certain heights, and all other vertices `fall' naturally according to their adjacencies with the pinned vertices.  For example, in the case of the fully extended Adinkras in Figures \ref{fig:small_adinkra} and \ref{fig:basic_heights}, only one vertex is a pinned vertex, secured at the topmost height, and all other vertices' heights are determined by their distance from that vertex. In the valise Adinkra, all 16 white vertices are pinned, secured at the same height, and thus all black vertices fall one step lower. 
\end{remark}

\begin{figure}[H]
\begin{tikzpicture}
    \node[anchor=south west,inner sep=0] at (-4,2) {\includegraphics[width=.6\textwidth]{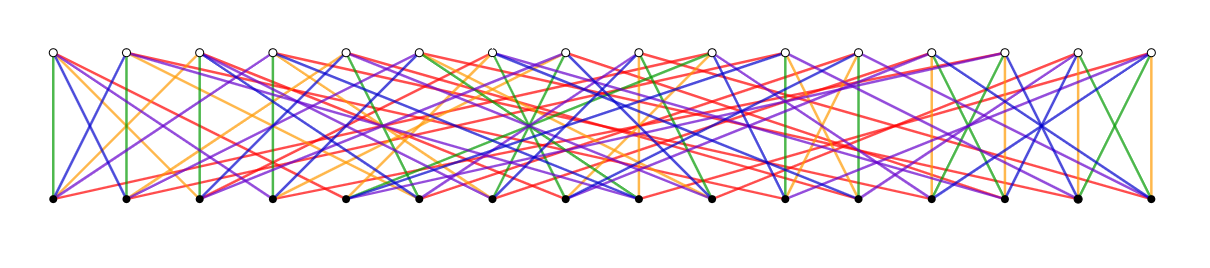}};
     \node[anchor=south west,inner sep=0] at (7,0) {\includegraphics[width=.3\textwidth]{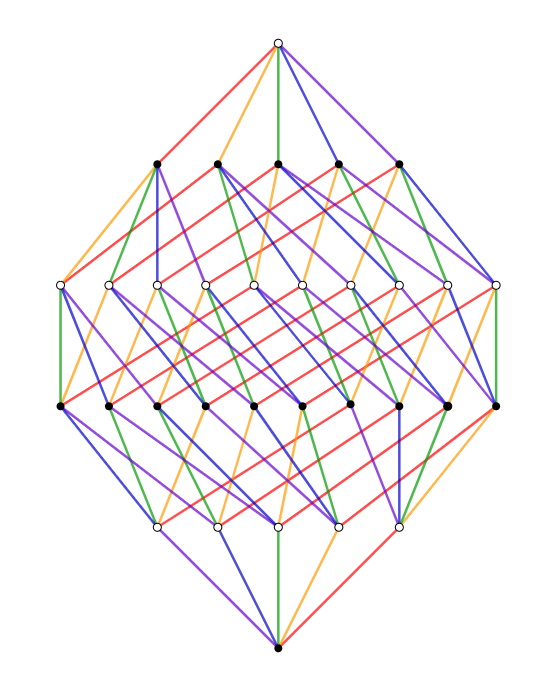}};

    \node at (1,-1) {$N=5$ Valise height};
    \node at (10,-1) {$N=5$ Fully extended height};

\end{tikzpicture}   

    \caption{Two extremal examples of heights on $H^5$, the valise height and the fully extended height.}
    \label{fig:basic_heights}
\end{figure}

\subsection{Counting and classifying Adinkra heights}\label{SS:CountingHeights}

As mentioned in the introduction, we are interested in characterizing the set of images of heights in the Jacobian of the algebraization of the Adinkra. First, we consider characterizing the heights themselves. In his thesis, Zhang \cite{Z13} presented an algorithm for counting heights on $H^5$ inductively. The basic idea of this algorithm is as follows. 
    \begin{itemize}
    \item Beginning with the $N = 1$ hypercube $H^1$,  list all possible heights. 
    \item Construct heights on $H^{N+1}$  by taking all possible pairs of heights on $H^N$, shifting one of them up or down one unit, and glueing them together.
    \item Check to see if the new height has already been counted.
    \end{itemize}
This is a computationally costly algorithm. For $N=1$ there are two heights (either the black or the white vertex on top). For $N=2,3,4,5$ there are $6,38,990,395094$ heights, respectively. 

An alternative way to obtain these numbers, as shown in \cite{Z13},  comes from the relationship between  heights on the $N$-hypercube $H^N$ and  3-colorings of  $H^N$.
\begin{proposition}\cite{Z13}
    The number of 3-colorings of $H^N$ is three times the number of distinct height functions on $H^N$. 
\end{proposition}
 In OEIS \cite{S}, the sequence of the number of 3-colorings of $H^N$ has been computed up to $N=6$.
 \begin{table}[H]
     \centering
        \begin{tabular}{c|l|l}
            $N$ & 3-colorings of $H^N$ &  heights of $H^5$\\\hline
            1  & 6 & 2 \\
            2 & 18 & 6\\
            3 & 114 & 38\\
            4 & 2970 & 990\\
            5 &1,185,282 & 395,094\\
            6 & 100,301,050,602 & 33,433,683,534\\
        \end{tabular}
    \caption{The number of 3-colorings and heights of $H^N$ for $N\leq 6$}
    \label{tab:3colorings}
 \end{table}
 
Some further work along these lines appeared in \cite{Gal}, where it was shown that for large $N$, the number of heights on $H^5$ is $O(2^{2^{N-1}})$ and the proportion of these heights with total spread equal to $3, 4, 5$ is roughly $\frac{1}{e}, 2\frac{\sqrt{e} - 1}{e}, 1-\frac{2\sqrt{e} - 1}{e}$, respectively. These results are relevant even for small  $N$ as seen in  Table \ref{tab:3colorings}.

Heights on $H^5$ can be grouped into equivalence classes using restricted graph isomorphisms in the following way. 
\begin{definition}\label{def:hequiv}
    We say that two heights $h$ and $h'$ on an Adinkra $A$ are {\it combinatorially equivalent}, denoted $h\, {\sim}_{\rm c}\, h'$ or $h \in [h']_{\rm c}$, if there exists a graph isomorphism $g\colon H^N \to H^N$ such that $h(v) = h'(g(v))$ for all $v \in V(A)$. 
\end{definition} 
Two interesting and, in some sense, extremal heights on a hypercube Adinkra are the fully extended height and the valise height, as seen in Figure \ref{fig:basic_heights}. It is straightforward to see that there are $2^N$ heights in the fully extended height equivalence class on $H^N$, and 2 heights in that of the valise.  Describing all of the equivalence classes of heights on $H^N$ under $\sim_{\rm c}$ is difficult and, to our knowledge, open for $n\geq 5$. Such a description would be very helpful for studying Adinkras.

\begin{remark}\label{rem:Gamma}
    The height functions on the $N$-hypercube $H^N$ can be placed in a directed graph $\Gamma_N$ in the following way. Given height functions $h$ and $h'$,  we say that $h \rightarrow h'$ if for some vertex $v \in V(H^N)$, 
 $h'|_{V(H^N)\setminus \{v\}}  = h$ and $h'(v) = h(v)-2$ (in this case $v$ is a local max in $h$ and a local min in $h'$). In other words, adjacency in $\Gamma_N$ is given by vertex lowering.
\end{remark}
Note also that the ``reduced'' digraph $\widetilde \Gamma_N$ of height equivalence classes $[h]_{\rm c}$ (where, once again, adjacency is given by vertex-lowering) is well-defined. 
The details of the number and digraph structure of these equivalence classes is a nontrivial problem. In future work, we describe $\widetilde\Gamma_N$, for $N = 1, 2, 3, 4$, and discuss $\widetilde \Gamma_5$.

\subsection{Discrete Morse functions and Morse divisors}\label{S:discreteMorse}
In Section \ref{S:Geometrization} we will see how to embed an Adinkra $A$ into an certain algebraic curve $X_{\rm alg}^N$ in $\mathbb{P}^N$. We are interested in how information about Morse divisors on $X_{\rm alg}^N$ and the Jacobian of $X_{\rm alg}^N$, associated with height functions on $A$, can help us understand Adinkras, and by extension, 1-dimensional supersymmetry algebras. 
In fact, finding the image of a height as a divisor of  ${\rm Jac}(X_{\rm alg})^N$, is a purely combinatorial process, which we will explain now. To do so, we will need to use the notion of {\it discrete Morse functions}, following \cite{Banchoff, geom2}. 

We refer to an oriented homogeneous simplicial 2-complex as a \emph{triangular mesh}. A discrete Morse function $f$ on a triangular mesh $M$ is a real-valued function defined on the vertices of $M$, such that adjacent vertices are mapped to distinct values. We convert a hypercube Adinkra $H^N$ equipped with a rainbow and a choice of one vertex as a \emph{base point} into a triangular mesh $\widetilde H^N$ in the following way. 

We begin by constructing an oriented 2-dimensional cell complex. Consider the topological space consisting of the vertices and edges of the Adinkra together with a collection of 2-cells formed in the following way: we add one 2-cell for each 4-cycle in the Adinkra formed by consecutive colors $(j, j+1)$ in the rainbow. We orient each of these 2-cells by following the rule that if you start at a vertex of the boundary 4-cycle that is an even distance from the base point, you should traverse the edges of the 4-cycle by starting with the edge of color $j$, while if you start at a vertex of the 4-cycle that is an odd distance from the base point, you should start with the edge of color $j+1$. (In the case where there are only 2 edge colors, we pick one color as the starting color for even distance and the other color as the starting color for odd distance.)

To convert our cell complex to a triangular mesh, we must triangulate the 2-cells.  To do so, we use the height function on the Adinkra. Each of the $(j, j+1)$ 4-cycles $C$ in the Adinkra is either a {\it diamond} (if $|\{h(v) \mid v \in C\}|=3$) or a bow tie (if $|\{h(v) \mid v \in C\}|=2$)).

\begin{figure}[H]
\begin{tabular}{cc}
\includegraphics[scale=.3]{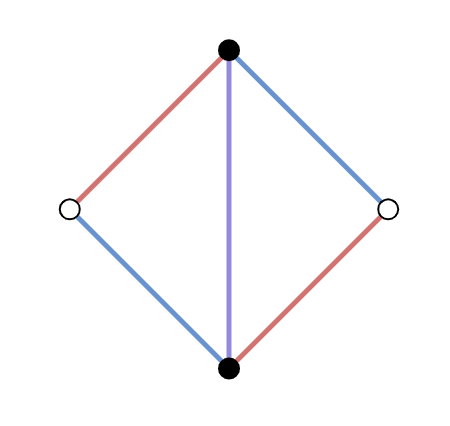} & 
\includegraphics[scale=.3]{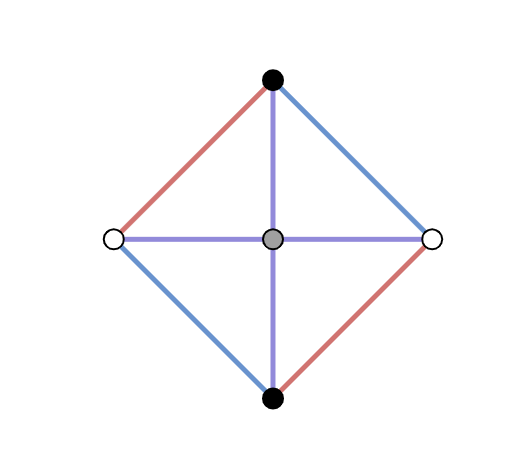}
\end{tabular}
\caption{A triangulated diamond 4-cycle and bow-tie 4-cycle}\label{fig:triangulate}
\end{figure}

In the case of a diamond, we triangulate by adding an edge that connects the maximum vertex and minimum vertex in the cycle. For each bow-tie cycle, we add a new vertex    adjacent to all 4 vertices in the cycle, with height equal to the average of the heights of these vertices. (See Figure \ref{fig:triangulate}). In such a case,  let $f_{j,j+1}(v)$ denote the face center   in the $(j,j+1)$-colored 4-cycle adjacent to $v$. The 2-simplices resulting from our triangulation inherit an orientation from the orientation on the initial cell complex.

\begin{definition}\label{def:dMD}
    Given a discrete Morse function $f$, a {\it discrete Morse divisor} $D$ is a formal sum $D = \sum_v \kappa_v v$ of the vertices $v$ in the triangular mesh, with coefficients $\kappa_v$ determined in the following way. For each $v$, we consider the set of all vertices $u_j$ on the boundary of $\mathrm{Star}(v)$, with the cyclic ordering induced by the orientation on the triangular mesh. We start at any such $u_j$ and travel along the edges to $u_{j+1}$ (following the order dictated by the orientation of our mesh), counting the number $\lambda_v$ of times that the value $f(u_j)-f(u_{j+1})$ changes sign in a single traversal of the boundary of $\mathrm{Star}(v)$. (Note that $\lambda_v$ is always even.) For $v$ a local minimum or maximum of $f$, $\lambda_v = 0$ and $\kappa_v = -1$. If $\ell_v = 2$, $v$ is called a regular point and $\kappa_v = 0$. If $\lambda_v = 2 + 2\mu_v$ for $\mu_v>0$, then $v$ is called a saddle point and $\kappa_v = \mu_v$.
\end{definition}

Note that choosing the opposite orientation for our triangular mesh (as might happen, for example, if we picked a fermion instead of a boson as our base point) does not change the discrete Morse divisor associated to a height function, because Definition~\ref{def:dMD} only counts changes in sign.

Below, we will use the notion of a discrete Morse function to compute discrete Morse divisors for four distinct heights on $H^5$. Three of these heights ($h_{\rm v}$, $h_{\rm fe}$, and $h_{1}$) were partially treated in \cite{geom2}. We give complete computations in these cases, and present a new example as well. 

\begin{example}\label{ex:divisors}
\noindent {\bf (a)} Let $h_{\rm v}$ be the valise height on $H^5$, as shown on the left in Figure \ref{fig:basic_heights}. For this height every vertex is a local max or min, and every face center $f_{j,j+1}(v)$ is a bow-tie, and thus, a saddle point with multiplicity $\mu_v = 1$. The associated discrete Morse divisor is 

\[
D_{\rm v} = \sum_{f \text{ face center}} f - \sum_{v \text{ vertex}} v,
\]

{\bf (b)} Let $h_{\rm fe}$ be the fully extended height on $H^5$, as shown on the right in Figure \ref{fig:basic_heights}.  Note that in this case  all face centers of the adinkra are diamonds, thus no face centers appear in $D_{\rm fe}$. Without loss of generality, we may assume that $(1,1,1,1,1)$ is the single pinned vertex of $H^5$. Then, the height of any black or white vertex can be taken to be the sum of the entries of its representation in   $H^5$. It is clear that we have only one maximum, $(1,1,1,1,1)$, and one minimum, $(0,0,0,0,0)$. To find saddle points,  consider the diagrams in Figure \ref{fig:saddle points}; these demonstrate the two ways (modulo rotating by $2\pi/5$) in which a saddle point can occur; in both cases the saddle point is the central vertex, which has height $h_0$. 

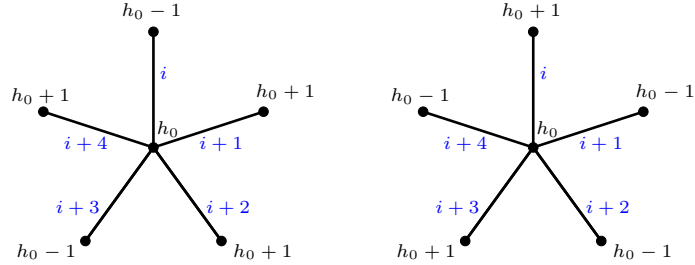
\begin{figure}[H]
\begin{center}
    \begin{tikzpicture}[line cap=round,line join=round,>=triangle 45,x=1cm,y=1cm,scale = 1.5]
\draw [line width=1pt] (2.5,-0.46115823141237366)-- (2.5,0.5647362809793773);
\draw [line width=1pt] (1.8969943352083505,-1.2911243263732994)-- (2.5,-0.46115823141237366);
\draw [line width=1pt] (2.5,-0.46115823141237366)-- (1.524316338958386,-0.14413939264732145);
\draw [line width=1pt] (3.1030056647916484,-1.2911243263733005)-- (2.5,-0.46115823141237366);
\draw [line width=1pt] (2.5,-0.46115823141237366)-- (1.8969943352083505,-1.2911243263732994);
\draw [line width=1pt] (3.475683661041614,-0.1441393926473227)-- (2.5,-0.46115823141237366);
\draw [line width=1pt] (2.5,-0.46115823141237366)-- (3.1030056647916484,-1.2911243263733005);
\begin{scriptsize}

\draw [fill=black] (2.5,-0.46115823141237366) circle (1.2pt);
\draw[color=black] (2.63,-0.3) node {$h_0$};
\draw [fill=black] (2.5,0.5647362809793773) circle (1.2pt);
\draw[color=black] (2.483478086269523,0.7463646395605015) node {$h_0-1$};
\draw [fill=black] (3.475683661041614,-0.1441393926473227) circle (1.2pt);
\draw[color=black,anchor=south] (3.67,-0.1) node {$h_0+1$};
\draw [fill=black] (1.524316338958386,-0.14413939264732145) circle (1.2pt);
\draw[color=black,anchor = south] (1.5,-0.14) node {$h_0+1$};
\draw [fill=black] (1.8969943352083505,-1.2911243263732994) circle (1.2pt);
\draw[color=black,anchor = north] (1.55,-1.25) node {$h_0-1$};
\draw [fill=black] (3.1030056647916484,-1.2911243263733005) circle (1.2pt);
\draw[color=black,anchor = north] (3.47,-1.25) node {$h_0+1$};
\draw[color=blue,anchor = west] (2.49,.2) node {$i$};
\draw[color=blue,anchor = north] (3.1,-.3) node {$i+1$};
\draw[color=blue,anchor = west] (2.9,-1) node {$i+2$};
\draw[color=blue,anchor = east] (2.09,-1) node {$i+3$};
\draw[color=blue,anchor = north] (1.9,-.3) node {$i+4$};
\end{scriptsize}
\end{tikzpicture}
\qquad
    \begin{tikzpicture}[line cap=round,line join=round,>=triangle 45,x=1cm,y=1cm,scale = 1.5]
\draw [line width=1pt] (2.5,-0.46115823141237366)-- (2.5,0.5647362809793773);
\draw [line width=1pt] (1.8969943352083505,-1.2911243263732994)-- (2.5,-0.46115823141237366);
\draw [line width=1pt] (2.5,-0.46115823141237366)-- (1.524316338958386,-0.14413939264732145);
\draw [line width=1pt] (3.1030056647916484,-1.2911243263733005)-- (2.5,-0.46115823141237366);
\draw [line width=1pt] (2.5,-0.46115823141237366)-- (1.8969943352083505,-1.2911243263732994);
\draw [line width=1pt] (3.475683661041614,-0.1441393926473227)-- (2.5,-0.46115823141237366);
\draw [line width=1pt] (2.5,-0.46115823141237366)-- (3.1030056647916484,-1.2911243263733005);
\begin{scriptsize}

\draw [fill=black] (2.5,-0.46115823141237366) circle (1.2pt);
\draw[color=black] (2.63,-0.3) node {$h_0$};
\draw [fill=black] (2.5,0.5647362809793773) circle (1.2pt);
\draw[color=black] (2.483478086269523,0.7463646395605015) node {$h_0+1$};
\draw [fill=black] (3.475683661041614,-0.1441393926473227) circle (1.2pt);
\draw[color=black,anchor=south] (3.67,-0.1) node {$h_0-1$};
\draw [fill=black] (1.524316338958386,-0.14413939264732145) circle (1.2pt);
\draw[color=black,anchor = south] (1.5,-0.14) node {$h_0-1$};
\draw [fill=black] (1.8969943352083505,-1.2911243263732994) circle (1.2pt);
\draw[color=black,anchor = north] (1.55,-1.25) node {$h_0+1$};
\draw [fill=black] (3.1030056647916484,-1.2911243263733005) circle (1.2pt);
\draw[color=black,anchor = north] (3.47,-1.25) node {$h_0-1$};
\draw[color=blue,anchor = west] (2.49,.2) node {$i$};
\draw[color=blue,anchor = north] (3.1,-.3) node {$i+1$};
\draw[color=blue,anchor = west] (2.9,-1) node {$i+2$};
\draw[color=blue,anchor = east] (2.09,-1) node {$i+3$};
\draw[color=blue,anchor = north] (1.9,-.3) node {$i+4$};

\end{scriptsize}
\end{tikzpicture}
\end{center}
\caption{Possible saddle point configurations in $X_{\rm alg}^5$. Vertices are labeled by their heights, relative to the height $h_0$ of the central vertex. Edges are labeled by their colors.}\label{fig:saddle points}
\end{figure}

Recall that adjacency (in $H^5$) along an edge of color $i$ corresponds to adding $1\pmod{2}$  to coordinate $k$. Thus, saddle points must look like $(1,0,1,0,1)$ or $(0,1,0,1,0)$ (modulo cyclic permutations of the coordinates). Then, denoting by $D_{\rm fe}$ the discrete Morse divisor associated with the height $h_{\rm fe}$,
\[
D_{\rm fe} = \sum_{v \in \mathcal{V}} v - (1,1,1,1,1) - (0,0,0,0,0),
\]
where 
\[
\mathcal{V} = \left\{
\begin{array}{c}(1,0,1,0,1),(0,1,0,1,1),(1,0,1,1,0),(0,1,1,0,1),(1,1,0,1,0),\\(0,1,0,1,0), (1,0,1,0, 0), (0,1,0, 0,1), (1,0, 0,1,0), (0, 0,1,0,1)\end{array}\right\}.
\]

\noindent {\bf (c)} Consider instead the height $h_1$ obtained from $h_{\rm fe}$ by lowering the vertex $(1,1,1,1,1) \in H^5$. We now have 5 new local maxima, which form the set $\mathcal{V}'$:
\[
\mathcal{V}' = \{
(0,1,1,1,1),(1,0,1,1,1),(1,1,0,1,1),(1,1,1,0,1),(1,1,1,1,0)\},
\]
and 5 new bow-tie type saddle points, which form the set $\mathcal{V}''$:
\[
\mathcal{V}'' = \left\{
\left(\frac12,\frac12,1,1,1\right),\left(1,\frac12,\frac12,1,1\right),\left(1,1,\frac12,\frac12,1\right),\left(1,1,1,\frac12,\frac12\right),\left(\frac12,1,1,1,\frac12\right)\right\}.
\]
Then 
\[
D_{h_1} =\sum_{ v \in \mathcal{V}}v + \sum_{v \in \mathcal{V}''}v-\sum_{v\in \mathcal{V}'}v-(0,0,0,0,0)-(1,1,1,1,1).
\]

When no confusion will arise, we will also refer to this divisor as $D_1$.

\noindent {\bf (d)} Finally,  consider the height $h_2$ obtained from $h_{1}$ by lowering the vertex $(0,1,1,1,1) \in H^5$. For this height,  $(1,1,1,1,1)$ is no longer a min or max; $(0,1,0,1,1)$ and $(0,1,1,0,1)$ must be removed from $\mathcal{V}$, and $\left(\frac 12, \frac 12, 1,1,1 \right)$ and $\left(\frac 12, 1, 1,1,\frac 12\right)$ must be removed from $\mathcal{V}''$. Three additional saddle points $\left(0,\frac12, \frac12,1,1\right)$, $\left(0,1,\frac12, \frac12,1\right)$, $\left(0,1,1,\frac12, \frac12\right)$, will be added. Thus, 
\begin{multline*}
D_{h_2} = \sum_{ v \in \widetilde{\mathcal{V}}}v+ \left(0,\frac12, \frac12,1,1\right)+\left(0,1,\frac12, \frac12,1\right)+\left(0,1,1,\frac12, \frac12\right)\\
+\left(1,\frac12,\frac12,1,1\right)+\left(1,1,\frac12,\frac12,1\right)+\left(1,1,1,\frac12,\frac12\right)\\
 -\sum_{v\in \mathcal{V}'}v-(0,0,0,0,0), \qquad\qquad
\end{multline*}
where
\[
\widetilde{\mathcal{V}} = \left\{
\begin{array}{c}(1,0,1,0,1),(1,0,1,1,0),(1,1,0,1,0),(0,1,0,1,0),\\ (1,0,1,0, 0), (0,1,0, 0,1), (1,0, 0,1,0), (0, 0,1,0,1)\end{array}\right\}.
\]
When no confusion will arise, we will also refer to this divisor as $D_2$.
\end{example}

It is natural to ask questions about which height functions give divisors that are equivalent in some sense. We discuss progress and open problems along these lines in future work.

\subsection{The combinatorics of color splitting}\label{SS:csplit}

One may use an edge color in a rainbow to label the vertices of an Adinkra in a new way.

\begin{algorithm}\label{A:colorSplit}
Given an Adinkra $A$ equipped with a rainbow, fix a color $j$ in the Adinkra rainbow and a starting vertex $v$. Consider the pair of colors $(j,j+2)$, which are adjacent to $j+1$ in the Adinkra rainbow. (Here, we use the cyclic ordering on the rainbow.) For each vertex $w \in A$, label $w$ with a $+$ sign if every path from $v$ to $w$ uses an even number of edges of colors $j$ and $j+2$, label $w$ with a $-$ sign if every path from $v$ to $w$ uses an odd number of edges of color $j$ and $j+2$, and label $w$ with $\times$ if there are distinct paths from $v$ to $w$ using both odd and even numbers of edges of color $j$ and $j+2$.
\end{algorithm}

\begin{definition}\label{def:jcs}
Let $A$ be an Adinkra equipped with a rainbow and fix a vertex $v$ of $A$. We say $A$ \emph{admits a $j$-color-splitting} if Algorithm~\ref{A:colorSplit} using the color $j$ labels every vertex of $A$ with a $+$ or $-$. In this case, we refer to the map $S_{j,v}\colon A \to \{+,-\}$ that partitions  the vertices using their $+$ or $-$ labels as the \emph{$j$-color-splitting} with starting vertex $v$.
\end{definition}

If $A$ admits a $j$-color-splitting with starting vertex $v$, then $A$ also admits a $j$-color-splitting with any other starting vertex $v'$. If $v$ and $v'$ had the same labels under the $j$-color-splitting with starting vertex $v$, then the two  $j$-color-splittings are identical; if they had opposite labels, then the $+$ and $-$ labels in the two splittings are reversed.
We illustrate an $N=3$ hypercube Adinkra $H^3$ with the rainbow $(\text{green}, \text{orange}, \text{blue})$ and a 1-color-splitting (that is, a green-color-splitting) using the bottom vertex as starting vertex in Figure~\ref{F:colorSplit}. 

\begin{figure}[H]
\scalebox{.22}{\includegraphics{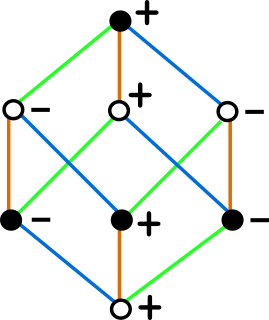}}
\caption{An $N=3$ hypercube Adinkra with an orange color-splitting}\label{F:colorSplit}
\end{figure}

\begin{proposition}\label{prop:hypercube_color_splittings}
Let $H^N$ be an $N$-dimensional hypercube Adinkra  with a rainbow $(1,\dots,N)$. Then for any starting vertex $v$ of $H^N$, $H^N$ admits a $j$-color-splitting by every color $j$ in the rainbow, each with $2^{N-1}$ vertices labeled $+$ and $2^{N-1}$ vertices labeled $-$.
\end{proposition}

\begin{proof}
We place specific coordinates $(x_1, \dots, x_N) \in (\mathbb{Z}/2\mathbb{Z})^N$ on $A$ as follows. Place some vertex $v$ at the origin $(0, \dots, 0)$. Assign coordinates to the other vertices so that traveling on an edge of color $j$ corresponds to changing the $j$th coordinate. Let $s_j: (\mathbb{Z}/2\mathbb{Z})^N \to \mathbb{Z}/2\mathbb{Z}$ be given by $s_j(x_1, \dots, x_N) = x_{j}+x_{j+2}$. (Here we use the cyclic ordering on the rainbow, so $s_N = x_N+x_2$ and $s_{N-1} = x_{N-1}+x_1$.) Running Algorithm~\ref{A:colorSplit}, we see that a vertex with coordinates $(x_1, \dots, x_N)$ will be labeled with $+$ if its image under $s_j$ is 0, because an even number of the edges of color $j$ and $j+2$ have been used. Similarly, the vertex will be labeled with $-$ if the image under $s_j$ is 1. Since $s_j$ maps every vertex to either 0 or 1, the Adinkra admits a color-splitting. Because $s_j$ is a group homomorphism, there are equal numbers of vertices with $+$ and $-$ labels.
\end{proof}

In some cases, one of the $j$-color-splittings on the hypercube $H^N$ described in Proposition~\ref{prop:hypercube_color_splittings} will induce a $j$-color-splitting on an Adinkra whose chromotopology is the quotient of $H^N$ by a doubly even code. One need simply check whether the homomorphism $s_j$ defined in the proof of the proposition is constant on the cosets of the doubly even code. For example, each of the $j$-color-splittings of $H^4$ induces a $j$-color-splitting on the quotient of $H^4$ by the doubly even code generated by $\{(1,1,1,1)\}$.

We will use a special term for the case when there are $j$-color-splittings for every color:

\begin{definition}\label{def:cs}
    Let $A$ be an Adinkra that admits a $j$-color-splitting for every color $j$ in the rainbow of $A$. Let $S_{j,v}(u)$ be the label of vertex $u$ under the $j$-color-splitting with starting vertex $v$. Define $S_v\colon A \to \{+,-\}^N$ by $S_v(u) = (S_{1,v}(u), S_{2,v}(u), \ldots, S_{N,v}(u))$.  We call $S_v$ the {\it total-color-splitting} of $A$ with starting vertex $v$.
\end{definition}

\section{Geometrization and the Jacobian map}\label{S:Geometrization}

\subsection{Geometrization for $N=5$}

The authors of \cite{geom1,geom2} showed that an Adinkra equipped with a rainbow determines an algebraic curve $X$ and a Bely\u{\i} map $X \to \mathbb{P}^1$. Topologically, the curve is given by gluing a 2-cell to each 4-cycle in the Adinkra where the edge colors are adjacent in the rainbow; in this way, the Adinkra is embedded in the curve. In the case of the $N$-hypercube Adinkra $H^N$, the authors of \cite{geom1,geom2} used results of~\cite{CHQ} to obtain a specific algebraic model $X^N_\mathrm{alg}$ for $X$ as a complete intersection in $\mathbb{P}^{N-1}$. 

Just as general Adinkras can be realized as quotients of hypercube Adinkras, the corresponding curves can be realized as quotients of the hypercube curves. For $N\geq 2$, the genus of the curve corresponding to the quotient of the $N$-hypercube by a $k$-dimensional doubly even code is $g = 1+2^{N-k-3}(N-4)$. In the case of $N$-hypercube curves, we see that $N=2$ and $N=3$ correspond to genus 0 curves, for $N=4$ we have a genus 1 curve, and for $N=5$ we have a curve of genus 5. As the $N=5$ case presents a significant increase in geometric complexity, we wish to examine it more closely. In the following, we restrict our discussion to this case. 

Equipping $H^5$ with a rainbow determines an algebraic curve $X$ with an embedded copy of $H^5$ and a Bely\u{\i} map $\beta \colon X \to \mathbb{P}^1$, ramified over $\{0,1,\infty\}$, such that $\beta(H^5) = [0,1]$, the black vertices are mapped to 0, the white vertices are mapped to 1, and the face centers are mapped to $\infty$. In fact, this map factors through the map $\overline \beta\colon B_5 \to \mathbb{P}^1$, given by 
$$
\overline \beta(z) = \frac{z^5}{z^5+1},
$$ 
where $B_5 \cong \mathbb{C}\mathbb{P}^1$ is called a {\it beach ball} (visualized in the left pane of Figure~\ref{fig:alphas}). Thus, the 5th roots of $-1$ are the preimages under $\overline\beta$ of $\infty$, and lines through the origin and 5th roots of unity are mapped to $[0,1]$.  
Let  $\zeta   = \exp^{\frac{2\pi i}{10}} $, and let $\eta$ be the M\"obius transformation given by
$$
\eta(x) = \frac{x - \zeta}{x - \zeta^{-1}} \cdot \frac{\zeta^3- \zeta^{-1}}{ \zeta-\zeta^3 }.
$$
Following \cite{geom1}, we obtain a list of points $\alpha_1, \dots, \alpha_5$ in $\mathbb{P}^1 = \mathbb{C} \cup \infty$ by setting $\alpha_i =-\eta(\zeta^{2i-1}) $. Note that we have $\alpha_1 = 0$, $\alpha_2 = 1$, and $\alpha_5 = \infty$. Additionally, we define $\alpha_0 = \eta(0)$, and $\alpha_\infty = \eta(\infty)$. 
See the right pane of Figure \ref{fig:alphas} for  locations of $\alpha_0$ and $\alpha_\infty$ relative to  $-\alpha_j$ ($j = 1,2,3,4,5$) in the complex plane.

In the language of \cite{CHQ}, the $N=5$ curve $X^5_\mathrm{alg}$ is a generalized Fermat curve of type $(2,4)$. Such a curve is given by a complete intersection of the form
\begin{equation}\label{eq:completeint} \begin{cases}
	x_1^2 + x_2^2 + x_3^2 &= 0 \\
	\alpha_3 x_1^2 + x_2^2 + x_4^2 &= 0 \\
	\alpha_4 x_1^2 + x_2^2 + x_5^2 &= 0 \\
\end{cases}.
\end{equation}

\noindent To specify $X^5_\mathrm{alg}$, we take $\alpha_3$ and $\alpha_4$ as defined above. In fact, it is straightforward to verify that $\alpha_3 = \phi$ and $\alpha_4 = 1+\phi$, for $\phi = (1+\sqrt{5})/2 \approx 1.618$, the golden ratio.

Now consider the projection $\pi: X_\mathrm{alg}^5 \to \mathbb{P}^1$ given by $[x_1: \ldots : x_5] \mapsto [{x_1^2}:{x_2^2}]$. Note that 
\begin{equation}\label{eq:invpi}
\begin{array}{c}
\pi^{-1}([1:x]) = [1: \pm \sqrt{x} : \pm \sqrt{-1 -x}: \pm \sqrt{-\alpha_3 -x}: \pm \sqrt{-\alpha_4 -x}],\\[2mm]
\pi^{-1}([0:1]) = [0: \pm i : \pm 1: \pm 1: \pm 1].
\end{array}
\end{equation}
The relationship between the various spaces described is visualized in  the diagram below. 
\[
  \begin{tikzcd}
  &&&&\\
    H^5 \arrow[r, hook]  & X_{\rm alg}^5  \arrow[dr ,"\pi"] \arrow[r, "\pi\circ \eta^{-1}" ] \arrow[rr, "\beta", bend left]
    & B_5 \arrow[d,"\eta"]\arrow[r,"\overline\beta"] & \mathbb{CP}^1 \\
    &&\eta(B_5) \arrow[ur,swap, "\overline\beta \circ \eta^{-1}" ]& &&
  \end{tikzcd}
\]

We can use the map $\pi$ to describe the embedding of the $N=5$ hypercube Adinkra $H^5$ in $X^5_\mathrm{alg}$. Note that $\pi(X_{\rm alg}^5)$ is the image of $B_N$ under $\eta$ (see Figure \ref{fig:alphas}). Thus, $\pi^{-1}(\alpha_0)$ and $\pi^{-1}(\alpha_\infty)$ will coincide with the black and white vertices, respectively.  
The $2^4$ vertices of each color correspond to the $2^4$ sign choices in Equation~\ref{eq:invpi}. Similarly, $\pi^{-1}(\alpha_i)$ will coincide with the 8 possible face centers in $X_{\rm alg}^5$.

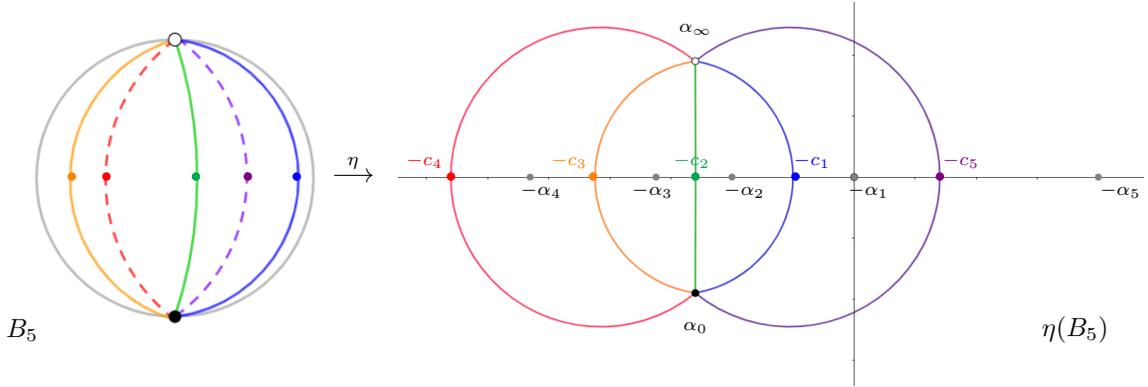
\begin{figure}[H]
    \centering
    \begin{tikzpicture}
        \node[anchor=south west,inner sep=0] at (-5.2,.7){\includegraphics[width=.27\textwidth]{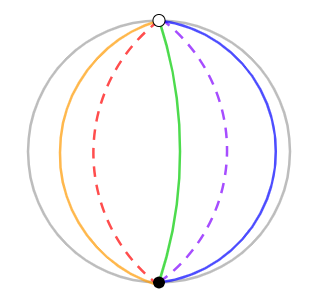}};
    \node at (-.6,2.9) {$\overset{\eta}{\longrightarrow}$};
    \begin{scriptsize}
        
    \filldraw[blue] (-1.35,2.79) circle (1.4pt);
    \filldraw[Green] (-2.69,2.79) circle (1.4pt);
    \filldraw[orange] (-4.34,2.79) circle (1.4pt);
    \filldraw[red] (-3.88,2.79) circle (1.4pt);
    \filldraw[violet] (-2,2.79) circle (1.4pt);

        \end{scriptsize}
    \node at (-5,.75) {$B_5$};
 
    \node[anchor=south west,inner sep=0] at (0,0) {\includegraphics[width=.6\textwidth]{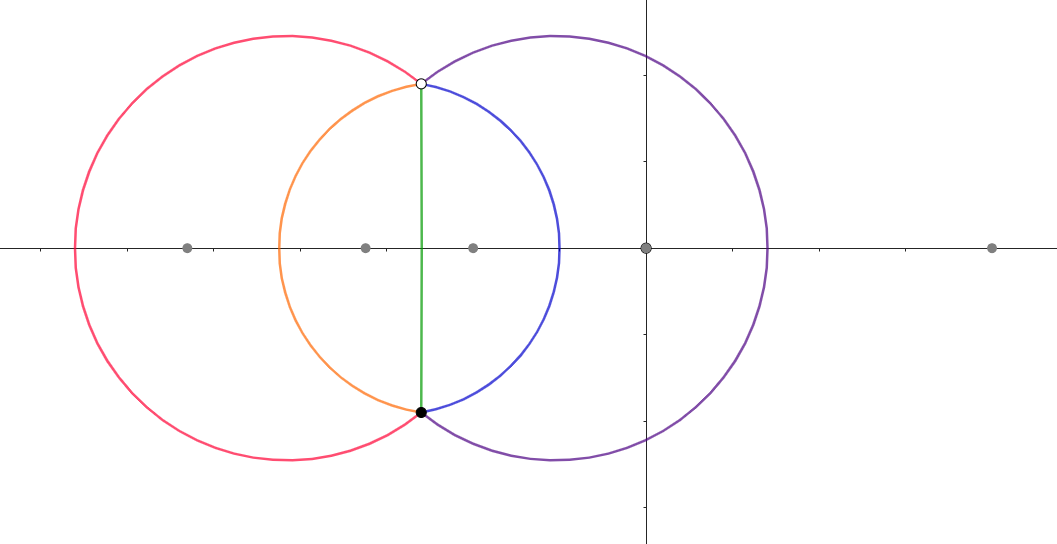}};
    \begin{scriptsize}
        
    \node at (6.25,2.6) {$-\alpha_1$};
    \node at (4.6,2.6) {$-\alpha_2$};
    \node at (3.38,2.6) {$-\alpha_3$};
    \node at (1.9,2.6) {$-\alpha_4$};
    \node at (9.6,2.6) {$-\alpha_5$};
    \node at (3.95,.78) {$\alpha_0$};
    \node at (3.95,4.75) {$\alpha_\infty$};
    \node[blue] at (5.5,3.) {$-c_1$};
    \filldraw[blue] (5.28,2.79) circle (1.4pt);
    \node[Green] at (3.9,3) {$-c_2$};
    \filldraw[Green] (3.95,2.79) circle (1.4pt);
    \node[orange] at (2.27,3) {$-c_3$};
    \filldraw[orange] (2.59,2.79) circle (1.4pt);
    \node[red] at (0.35,3) {$-c_4$};
    \filldraw[red] (0.7,2.79) circle (1.4pt);
    \node[violet] at (7.48,3) {$-c_5$};
    \filldraw[violet] (7.2,2.79) circle (1.4pt);
        \end{scriptsize}
    \node at (9,.75) {$\eta(B_5)$};

\end{tikzpicture}   
\caption{The points $-\alpha_i$ and edge images under $\pi$ in $\mathbb{P}^1$}
    \label{fig:alphas}
\end{figure}

We will embed the edges of the Adinkra in $X^5_\mathrm{alg}$ using the inverse images under $\pi$ of paths between $\alpha_0$ and $\alpha_\infty$ in $\C$, as follows. To the color $j$ in our rainbow $(1,\dots,5)$, we associate the  path $\gamma_j$ parametrized by $\eta(t \zeta^{2j})$, $t\in (0,\infty)$. This path travels from $\alpha_0$ to $\alpha_\infty$ and crosses the $x$-axis at a unique point $-c_j$. Note that $c_j>\alpha_j$ for $j=1,2,3,4$,   $c_j<\alpha_{j+1}$ for $j=1,2,3$,   while $c_5 < \alpha_1$. We thus obtain a rainbow of paths in the complex plane, as shown in the right pane of Figure~\ref{fig:alphas}. Each edge of color $j$ in $X^5_\mathrm{alg}$ is then a connected component of $\pi^{-1}(\gamma_j)$.

The points in $\pi^{-1}(-\alpha_{j+1})$ are called face centers and denoted by $\mathbf{f}_{j,j+1}(f)$. There are 8 face centers corresponding to each pair of edge colors, as shown in Equation~\eqref{eq:faceCenter}.

\begin{align}\label{eq:faceCenter}
\begin{split}
\mathbf{f}_{5,1}\in \pi^{-1}(-\alpha_1) &= \pi^{-1}(0) =[1 : \, 0 :\pm i : \pm  i\sqrt{\phi} : \pm i\sqrt{\phi + 1}] \\
\mathbf{f}_{1,2}\in \pi^{-1}(-\alpha_2)& = \pi^{-1}(-1)  = [1 :\pm i : \, 0 :\pm  i\sqrt{\phi-1}: \pm i\sqrt{\phi }]\\
\mathbf{f}_{2,3}\in \pi^{-1}(-\alpha_3) & = \pi^{-1}(-\phi) = [1 :\pm i\sqrt{\phi}: \pm  \sqrt{\phi-1} : \, 0 : \pm 
i] \\
\mathbf{f}_{3,4}\in \pi^{-1}(-\alpha_4)& = \pi^{-1}(-\phi-1)  = [1 :\pm i\sqrt{\phi + 1} : \pm \sqrt{\phi} : \pm 1 :\, 
0] \\
\mathbf{f}_{4,5}\in \pi^{-1}(-\alpha_5)& = \pi^{-1}(\infty) = [0 :\, 1 :\pm i: \pm i :\pm i] 
\end{split}
\end{align}

We can use the points $c_j$ to determine vertex adjacency in $X^5_\mathrm{alg}$ along an edge of a specified color in the following way.
Consider the edge of color $j$ adjacent to a given (WLOG) white vertex 
$${\bf w} =[1:w_1:w_2:w_3:w_4:w_5]=[1: \pm \sqrt{\alpha_\infty} : \pm \sqrt{-1 -\alpha_\infty}: \pm \sqrt{-\alpha_3 -\alpha_\infty}: \pm \sqrt{-\alpha_4 -\alpha_\infty}].$$
 For $k = 2, \ldots, N$, let $r_{k}(j)$ be the radicand  of the $k$th coordinate $w_{k}$ of $w$, evaluated at $-c_j$ in place of $\alpha_\infty$. For example, $r_2(j) = -c_j$, $r_3(j) = -1+c_j$, etc. Then, the unique edge of color $j$ with one endpoint equal to white vertex ${\bf w}$, has as its other endpoint the black vertex
 $$
 {\bf b} = [1: {\rm sign}(r_{2}(j))\cdot \overline{w}_2:{\rm sign}(r_{3}(j)) \cdot\overline{w}_3:{\rm sign}(r_{4}(j))\cdot\overline{w}_4:{\rm sign}(r_{5}(j))\cdot \overline{w}_5].
 $$  

We summarize our description of the Adinkra embedding in the following lemmas.

\begin{lemma}\label{lem:adj}
The black and white vertices of $H^5$ are embedded in $X_\mathrm{alg}^5$ as 
\begin{align*}
{\bf b} &=[1: \pm \sqrt{\alpha_0} : \pm \sqrt{-1 -\alpha_0}: \pm \sqrt{-\alpha_3 -\alpha_0}: \pm \sqrt{-\alpha_4 -\alpha_0}]
&=&[1: \pm \zeta^8\sqrt{\phi}  : \pm \zeta :  \pm i \zeta^9 : \pm \zeta^2 \sqrt{\phi}  ], \text{ and}\\
{\bf w} &=[1: \pm \sqrt{\alpha_\infty} : \pm \sqrt{-1 -\alpha_\infty}: \pm \sqrt{-\alpha_3 -\alpha_\infty}: \pm \sqrt{-\alpha_4 -\alpha_\infty}]
&=&[1: \pm \zeta^2\sqrt{\phi}  : \pm \zeta^9 :  \pm i \zeta : \pm \zeta^8 \sqrt{\phi}  ],    
\end{align*}
respectively. A black vertex ${\bf b}$ connected to a white vertex ${\bf w}=[1:w_1:w_2:w_3:w_4:w_5]$ by an edge of color $j$ has coordinates ${\bf b} = [1: {\rm sgn}_2(j)\cdot \overline{w}_2 : {\rm sgn}_3(j)\cdot \overline{w}_3:{\rm sgn}_4(j) \cdot\overline{w}_4: {\rm sgn}_5(j)\cdot \overline{w}_5]$, where the sign ${\rm sgn}_k(j)$ is as given in Table~\ref{tab:adj}.
\end{lemma}

\begin{table}[H]
    \begin{tabular}{|c||c|c|c|c|c|}
    \hline
    ${\rm sgn}_k(j)$ &\multicolumn{5}{|c|}{edge color $j$}  \\\hline
$k$ & 1 & 2 & 3 & 4 & 5 \\\hline \hline
          1 & + & + & + & + &  + 
          \\\hline
          2 & $-$ & $-$ & $-$ & $-$ & + 
          \\\hline
          3 & $-$ & + & + & + &$-$ 
          \\\hline
         4 & $-$ & $-$ & + & + & $-$ 
         \\\hline
          5 & $-$ & $-$ & $-$ & + & $-$ 
          \\\hline
          
    \end{tabular}
    
    \vspace*{.3cm}
    
    \caption{Adjacency information for $X_{\rm alg}^5$. The $k$th coordinate of the unique neighbor ${\bf w}$ of ${\bf b}$ (or ${\bf b}$ of ${\bf w}$)  via an edge of color $j$ can be found by conjugating the $k$th coordinate of ${\bf b}$ (${\bf w}$) and multiplying by $\pm 1$, as indicated in column $j$, row $k$.}\label{tab:adj}
\end{table}

Recall from Equation~\eqref{eq:faceCenter}, that each coordinate of a face center point in ${\rm X}_{\rm alg}^5$ is either real or imaginary. This feature can be used to determine adjacency information for face center points as described in the following lemma. 

\begin{lemma}\label{lem:adj_fc}
 Suppose that $f$ is a saddle point in $\widetilde H^5$ adjacent to a black or white vertex $v$. Let $\mathbf{f}$ denote the image of this point in $X_{\rm alg}^5$, and similarly let $\mathbf{v}$ denote the image of $v $ in  $X_{\rm alg}^5$. Then the  $k$th coordinates $\mathbf{f}_k$ of $\mathbf{f}$ are  determined 
by the following equations; for each $k = 1, \ldots, 5$
\[
\begin{array}{rl}
    {\rm sign}({\rm Im}(\mathbf{v}_k))&= {\rm sign}({\rm Im}(\mathbf{f}_k)) \text{ if } \mathbf{f}_k \in i \mathbb{R}, \text{ and }\\[1mm]
    {\rm sign}({\rm Re}(\mathbf{v}_k))& ={\rm sign}({\rm Re}(\mathbf{f}_k))  \text{ if }  \mathbf{f}_k \in  \mathbb{R} . 
\end{array}
\]
\end{lemma}
\begin{proof}

The face centers of $H^5$ appear in $X_{\rm alg}^5$ as described in Equation~\eqref{eq:faceCenter}; in particular, we can assume each coordinate of a face center in $X_{\rm alg}^5$  is either strictly  real or strictly imaginary.

Let $L_{j, j+1} = \{x \in B_5 \mid \arg(x) \in (\zeta^{2j},\zeta^{2j+2})\}$ be the slice of $B_5$ captured by $\frac{2\pi\cdot 2j }{10}<\theta <\frac{2\pi\cdot 2(j+2) }{10}$, for $i = 1, 2, 3, 5$, let $L_j = \eta(L_{j,j+1})$ denote its image in $\eta(B_5)$. Let $L_4$ denote the image of $L_{4,5}$ under $\eta$ and the additional M\"obius transformation $z \mapsto 1/z$. These sets can be seen in the pane on the right of Figure~\ref{fig:alphas}  as cut out by the curves $\eta(t\zeta^{2j})$. Each component  of $\pi^{-1}(L_j)$ corresponds to one of the eight $(j,j+1)$-colored faces in $X_{\rm alg}^5$. 
Recall that 
\[
\begin{array}{rl}
\pi^{-1}([1:x]) &= [1: \pm \sqrt{x} : \pm \sqrt{-1 -x}: \pm \sqrt{-\phi -x}: \pm \sqrt{-(\phi+1) -x}], \\[1mm]
(\text{for the $j = 4$ case}) \quad 
\pi^{-1}([x:1]) &= [\pm  \sqrt{x}: 1: \pm \sqrt{-x-1}: \pm \sqrt{-\phi x-1}:\pm \sqrt{-(\phi+1)x-1}].
\end{array}
\]
Note that for $x$ and $x'$ in a path connected subset $K$ of $\mathbb{C}$,  ${\rm sign}({\rm Re}(\sqrt{x})) = {\rm sign}({\rm Re}(\sqrt{x'}))$ unless a path between $x$ and $x'$ in $K$ intersects the ray $\theta = \pi$. Similary, ${\rm sign}({\rm Im}(\sqrt{x})) = {\rm sign}({\rm Im}(\sqrt{x'}))$ unless a path between $x$ and $x'$ in $K$ intersects the ray $\theta = 0$.

For $j = 1,2, 3, 5$ , we consider the path-connected sets  
\begin{align}\label{eq:slicesets}
\begin{split}
    L_{j},& \\
-L_j-1& = \{-c-1 \mid c \in L_j\},\\
-L_j-\phi &= \{-c-\phi \mid c \in L_j\}, \\
-L_j-\phi-1 &= \{-c-\phi - 1 \mid c \in L_j\}
\end{split}
\end{align} 
in $\eta(B_5)$.
For the $j = 4$ case, we use instead the similarly defined
 sets 
 \begin{equation}\label{eq:slicesets4}
     L_4, \quad -L_4-1,\quad -\phi L_4-1,\quad -(\phi + 1)L_4-1
 \end{equation}
Table~\ref{tab:face_signs}  displays which of these sets intersect the rays $\theta = 0$ or $\theta = \pi$ for $j = 1, \ldots, 5$.     From this table we see, for example if $j = 2$, then sign of the imaginary parts of $\mathbf{f}_k$ and $\mathbf{v}_k$ must be the same for $k = 2, 5$, while the sign of the real parts of $\mathbf{f}_3$ and $\mathbf{v}_3$ must agree. Note that the real and imaginary parts of the 4th coordinates are not forced to agree, per Table~\ref{tab:face_signs}. For $j=4$, the sign of the imaginary components of $\mathbf{f}_k$ and $\mathbf{v}_k$ must be the same for $k=3,4,5$.

\begin{table}[H]
    \begin{tabular}{|c|c||c|c|c|c|c|}
    \hline
     &&\multicolumn{5}{|c|}{ $j$}  \\\hline
    {\rm coordinate}& {\rm set} &1 & 2 & 3 & 4 & 5 \\\hline \hline
    2&$L_{j}$ & $\pi$ & $\pi$ & $\pi$ & & 0, $\pi$   
    \\
    1& $L_4$ & & & & 0, $\pi$ & 
    \\\hline
    \multirow{ 2}{*}{3} &$-L_{j} - 1$ & $0, \pi$ & 0 & 0 &  &$\pi$ 
    \\
     &$-L_{4} - 1$ &  &  & & $\pi$ & 
    \\\hline
    \multirow{ 2}{*}{4} & $-L_{j}-\phi  $ & $\pi$ & 0, $\pi$ & 0 &  & $\pi$ 
    \\
     & $ -\phi L_4-1 $ &  &  &  & $\pi$ &  
    \\\hline
    \multirow{ 2}{*}{5} & $-L_{j}-\phi - 1  $ & $\pi$ & $\pi$ & 0, $\pi$ &  & $\pi$ 
    \\
     & $ ( -(\phi + 1)L_4-1)$ &  &  &  & $\pi$ &  
    \\\hline
          
    \end{tabular}
   \vspace*{.3cm}
    
    \caption{A summary of which of the rays $\theta = \pi$ and $\theta = 0$ intersect the path connected sets defined in Equations~\ref{eq:slicesets} and \ref{eq:slicesets4}.}\label{tab:face_signs}
\end{table}

\end{proof}

 The adjacency relationship among the face center points in $X_{\rm alg}^5$, can now be stated precisely: 

\begin{corollary}\label{cor:fc_geom}
Suppose that $v, v_j \in H^5$ are adjacent along color $j$, that ${f}_{k,k+1}(v)$ and ${f}_{k,k+1}(v_j)$ are saddle points in $\tilde H^5$ adjacent to $v$ and $v_j$, and that ${\bf f}_{k,k+1}$ and ${\bf f}'_{k,k+1}$ are the  images of ${f}_{k,k+1}(v)$ and ${f}_{k,k+1}(v_j)$, respectively in $X_{\rm alg}^5$.  Then the $\ell$th coordinates $x'_\ell$ and $x_\ell$ of ${\bf f}'_{k,k+1}$ and ${\bf f}_{k,k+1}$, respectively, satisfy $x'_\ell={\rm sgn}_\ell(j)\overline{x_\ell}$ (see Table~\ref{tab:adj}). 
\end{corollary}

\begin{remark}\label{rem:choice}
    Note that there is a non-canonical choice to be made  regarding the relationship between the vertices in $H^5$ and the points in $\pi^{-1}(\alpha_0)$ and $\pi^{-1}(\alpha_\infty)$. That is, any vertex $v \in H^5$ can be associated to any point in $\pi^{-1}(\alpha_0) \cup \pi^{-1}(\alpha_\infty)$. However, once any single assignment is made, the rest of the assignments are determined by rainbow and adjacency structures of both $H^5$ and $X_{\rm alg}^5$. 
    The choice (${\bf w}{(1,1,1,1,1)} =[1: \sqrt{\alpha_\infty} : \sqrt{-1 -\alpha_\infty}:  \sqrt{-\alpha_3 -\alpha_\infty}: \sqrt{-\alpha_4 -\alpha_\infty}]$) is made for the remaining examples in this paper. The results stemming from a different choice of ``base point'' in the Adinkra surface can be obtained from the computations done here using Table~\ref{tab:adj}. 
\end{remark}

 We demonstrate this choice of base point in the following example. 
 
\begin{example}\label{ex:1}
     First, we choose to associate the white vertex $(1,1,1,1,1)$ in $H^5$ to  the white vertex ${\bf w}{(1,1,1,1,1)}\in X_{\rm alg}^5$ in the following way:
    \[
    {\bf w}{(1,1,1,1,1)} =[1: \sqrt{\alpha_\infty} : \sqrt{-1 -\alpha_\infty}:  \sqrt{-\alpha_3 -\alpha_\infty}: \sqrt{-\alpha_4 -\alpha_\infty}]=[1:   \zeta^2\sqrt{\phi} :  \zeta^9 :   i\zeta^{6} :  \zeta^8\sqrt{\phi}].
    \]
     We can determine the locations of all vertices in $X_{\rm alg}^5$  using Lemma~\ref{lem:adj}. For example,
    \begin{align*}
    {\bf b}{(1,1,1,0,1)} &=[1: - \zeta^8\sqrt{\phi}  :  \zeta :  i \zeta^9 : - \zeta^2 \sqrt{\phi}  ],\\
    {\bf w}{(1,0,1,0,1)} &=[1:  \zeta^2\sqrt{\phi}  : - \zeta^9 :   i \zeta :  \zeta^8 \sqrt{\phi}  ],\\
    {\bf b}{(0,0,0,0,0)} &= [1:  \zeta^8\sqrt{\phi}  :  \zeta :  -i \zeta^9 :  \zeta^2 \sqrt{\phi}  ]. \quad
    \end{align*}

    Next, consider the $(1,2)$-colored 4-cycle containing the vertex $(1,1,1,1,1)$. Let $\mathbf{v} = \mathbf{v}(1,1,1,1,1)$. Table~\ref{tab:face_signs} tells us that for $\mathbf{f} = \mathbf{f}_{1,2}(1,1,1,1,1)$, the sign of the imaginary parts of $\mathbf{f}_k$ and $\mathbf{v}_k$ must be the same for $k = 2, 4, 5$. Thus, 
    \[
    \mathbf{f}   = [1 :i : 0 : - i\sqrt{\phi-1}:  -i\sqrt{\phi }]\in \pi^{-1}(-1)
    \]
\end{example}

\begin{remark}\label{rem:gMD}
Note that the discrete Morse divisors defined in Definition~\ref{def:dMD} can be viewed naturally on $X^N_{\rm alg}$, as formal sums of the corresponding points in $X^N_{\rm alg}$. Such divisors on $X^N_{\rm alg}$ will be called {\it geometric divisors} below.
\end{remark}

We now turn our attention to the Jacobian of $X^5_\mathrm{alg}$. Define involutions $a_1, \dots, a_5: \mathbb{P}^4 \to \mathbb{P}^4$ by
\begin{equation}(a_j(x))_k = \begin{cases}
    x_k & \text{if }k\neq j\\
    -x_k & \text{if }k = j
\end{cases}.\end{equation}
We have induced group actions on $X^5_\mathrm{alg}$ via Equation~\ref{eq:completeint}.
It was shown in \cite[\S 5.1]{CHQ} that the Jacobian $\Jac(X^5_\mathrm{alg})$ is isogenous to the product of the Jacobians of orbifolds $X^5_\mathrm{alg}/H$, where $H$ runs through the following groups:
\begin{gather}
H_1 = \langle a_1, a_2a_3, a_2a_4\rangle, \quad H_2 = \langle a_2, a_1a_3, a_1a_4\rangle, \quad
H_3 = \langle a_3, a_1a_2, a_2a_4\rangle, \\ 
H_4 = \langle a_4, a_2a_3, a_1a_2\rangle, \quad H_5 = \langle a_5, a_2a_3, a_2a_4\rangle.\notag
\end{gather}

Each of the resulting factors is an elliptic curve, yielding an isogenous decomposition
$\Jac(X_\mathrm{alg}^5) \equiv E_1 \times \dots \times E_5$, where the $E_k$ are given by the following equations. (We have followed the numbering of \cite[\S 5.1]{CHQ}.)
\begin{align}\label{eq:ECs}
\begin{split}
    E_1,E_4: y^2 &= x(x-1)(x-(\phi+1))\\
    E_2, E_5: y^2 &= x(x-1)(x-\phi)\\
    E_3: y^2 &= x(x-1)(x+\phi).
\end{split}
\end{align}

As noted in \cite[\S 5.2]{CHQ}, the elliptic curves $E_1, \dots, E_5$ have the same $j$-invariant (namely, 2048), and are thus isomorphic over $\mathbb{C}$. One may check that $E_1$, $E_2$, and $E_3$ are mutually non-isomorphic over $\mathbb{Q}(\phi)=\mathbb{Q}(\sqrt{5})$; in the \cite{LMFDB} database's list of curves over $\mathbb{Q}[\sqrt{5}]$, their isomorphism classes are labeled 256.1-c3, 256.1-b3, and 64.1-a3, respectively.

\subsection{Vertex and face center images}

We are now ready to characterize the image of $H^5$ in the elliptic curves $E_1, \dots, E_5$. We will prove the following theorem:

\begin{theorem}\label{thm:imagesinEC}
 Fix a rainbow for the $N=5$ hypercube Adinkra $H^5$.      The images of the face centers and vertices in $H^5$ in the Mordell-Weil group $MW(E_k)$ over $K_k$ lie in a group isomorphic to $\mathbb{Z} \times \mathbb{Z}/(4) \times \mathbb{Z}/(2)$.

\begin{enumerate}
    \item[A.] The projection of the image of a black or white vertex $u$ on $\mathbb{Z}$ is a generator of $\mathbb{Z}$. The signs of this projection are determined by a total-color-splitting  $S_v(u)$ of $H^5$ with $v = (1,1,1,1,1)$, with the sign of the projection associated to $MW(E_k)$ given by a $k$-color-splitting for each $k$. 
     \item[B.] The torsion subgroup is generated by the images of face centers; face centers corresponding to one pair of adjacent edge colors in the rainbow are mapped to the identity of $E_k$, face centers corresponding to another pair of adjacent edge colors in the rainbow are mapped to points of order 4 in $MW(E_k)$, and face centers corresponding to the remaining pairs of adjacent edge colors in the rainbow are mapped to points of order 2 in $MW(E_k)$.
\end{enumerate}
\end{theorem}

We begin by giving a map between $X_{\rm alg}^5$ and ${\rm Jac}(X_{\rm alg}^5) $, and then we use this map to describe the images of points in $H^5$ as points on elliptic curves over the complex numbers. Image points in $MW(E_k)$ of order greater than 2 are equipped with a choice of sign; we describe the relationship between these signs and adjacency in $H^5$ in Lemma~\ref{lem:color_split}. Appropriate fields of definition $K_k$ for the relevant elliptic curves are presented in Lemma~\ref{prop:field_choice}. 

Once these preliminaries are established, we prove Parts A and B of the above theorem by examining the group structure on the five elliptic curves $E_1, \ldots, E_5$. Our analysis yields a complete description of the map to $\mathbb{Z} \times \mathbb{Z}/(4) \times \mathbb{Z}/(2)$. In particular, we summarize the relationship between vertex or face center images and the abelian group $\mathbb{Z} \times \mathbb{Z}/(4) \times \mathbb{Z}/(2)$ in Table~\ref{tab:torsion_group}. Details of these computations can also be explored using the authors' SageMath code (see \cite{ourcode}).

Let us now consider the map $\nu= (\nu_1, \ldots, \nu_k) \colon X_\mathrm{alg}^5 \to E_k$ defined by 

\begin{gather}\label{eq:maps}
    \nu_k([x_1: \ldots: x_5]) = 
    \left\{ 
    \begin{array}{ll}
    \displaystyle E_1\Big(-(2\phi+1) \frac{x_1^2}{x_5^2} - \phi,\ (2\phi+1) \frac{x_2x_3x_4}{x_5^3}\Big) & 
    k = 1\\[2mm]
    \displaystyle E_2\Big( \frac{x_2^2}{x_1^2} +\phi+1,\ i \frac{x_3x_4x_5}{x_1^3}\Big) & 
    k = 2\\[2mm]
     \displaystyle E_3\Big((2\phi+1) \frac{x_3^2}{x_2^2} + \phi+1,\ i(2\phi+1) \frac{x_1x_4x_5}{x_2^3}\Big) & 
    k = 3\\[2mm]
     \displaystyle E_4\Big( (\phi+1)\frac{x_4^2}{x_3^2} - \phi,\ i\phi \frac{x_1x_2x_5}{x_3^3}\Big) & 
    k = 4\\[2mm]
     \displaystyle E_5\Big( -\frac{x_5^2}{x_4^2} + \phi+1,\ i \frac{x_1x_2x_3}{x_4^3}\Big) & 
    k = 5\\
    \end{array}
    \right. 
\end{gather}
Here, we use the notation $E_k(x,y)$ to designate the point $[x:y:1]$ on the curve $E_k$.  It is straightforward to verify that $\nu$ satisfies the conditions discussed following Remark~\ref{rem:gMD}, and thus gives the isogenous decomposition of $\Jac(X_\mathrm{alg}^5)$. All such maps  give isogenous images, so we are free to use $\nu$ to proceed. 
Geometric divisors are formal sums of the images of vertices and face centers, so it will be useful to have an explicit description of the images of the vertices and face centers in each of the elliptic curve components of the Jacobian. This is given in Tables \ref{tab:vertex_images} and \ref{tab:face_images}.

The images of the (black and white) vertices can be computed directly from the maps in Equation~\ref{eq:maps}. For a given map $\nu_k$, Table~\ref{tab:vertex_images} describes the elliptic curve points $B_k^\pm$ and $W_k^\pm$, where $W_k^+$ is taken to be the image of ${\bf w}(1,1,1,1,1)$ under $\nu_k$ and $B_k^+$ the image of ${\bf b}(0,0,0,0,0)$. 

\begin{table}[H]
    \begin{tabular}{|l|ll|}
    \hline
      $\underset{}{W_k^+}$ & $\overset{}{B_k^+}$ &\\\hline\hline
     $E_1(\zeta^8\phi,\  -i \zeta^8(\phi+1)) $&$\underset{}{E_1(\zeta^2\phi,\  -i \zeta^2(\phi+1))}$& $(=-\overline{W}_1^+)$  \\\hline
     $\underset{}{E_2(\zeta^9\phi,\  \zeta^2\sqrt{\phi})}$ & 
     $E_2(\zeta\phi,\  \zeta^8\sqrt{\phi})$ & $(=\overline{W}_2^+)$ \\\hline 
     $E_3(\zeta^2\phi,- \zeta^8(\phi+1)$ & $\underset{}{E_3( \zeta^8\phi,\  -\zeta^2(\phi+1))}$ & $(=\overline{W}_3^+)$ \\\hline
    $E_4(\zeta^8\phi,-i\zeta^8(\phi + 1))$
    &$\underset{}{E_4(\zeta^2\phi,\  -i\zeta^2(\phi + 1))}$ & $(=-\overline{W}_4^+)$ \\\hline
     $ E_5( \zeta\phi,  \zeta^8 \sqrt{\phi})$ &$ \underset{}{E_5( \zeta^9\phi,\  \zeta^2 \sqrt{\phi})}$& $(=\overline{W}_5^+)$\\\hline
    \end{tabular}
    \vspace*{.3cm}
    \caption{Vertex images on the elliptic curves using the maps in Equation \ref{eq:maps}.  For each $k$, $W_k^+$ is  $\nu_k([1: \zeta^2\sqrt{\phi}  :  \zeta^9 :   i \zeta : \ \zeta^8 \sqrt{\phi}  ])$ and $B_k^+$ is  $\nu_k([1: \zeta^2\sqrt{\phi}  :  \zeta^9 :   i \zeta^6 : \ \zeta^8 \sqrt{\phi}  ])$
    } 
    \label{tab:vertex_images}
\end{table}

We next turn our attention to face centers. For each pair of adjacent edge colors in the rainbow, there are 8 distinct face centers in $X^5_\mathrm{alg}$. However, under the maps described in Equation~\ref{eq:maps}, face centers corresponding to a single pair of edge colors are sent to either one or two points in each elliptic curve. We label these image points as $F_{(j,j+1)}$, using arithmetic $\pmod{5}$ on the indices to reflect the cyclic ordering on the rainbow. When we wish to distinguish points that arise as images of face centers corresponding to a single pair of edge colors, we use the notation $F_{(j,j+1)}^+$ and $F_{(j,j+1)}^-$, and we use $F_{(j,j+1)}^\pm$ to refer to either element of such a pair. 
Face centers corresponding to the pairs of colors $(k-1,k)$, $(k,k+1)$, and $(k+1,k+2)$  are sent to $E_k(0,0)$, $E_k(1,0)$, and $E_k(r,0)$, where $r$ depends on $k$ (c.f. Equation~\ref{eq:ECs}).  Face centers corresponding to colors $(k+2,k-2)$ are sent to the identity element $O_k$ in $E_k$, which we may think of as the point at infinity.  The face centers corresponding to the  edge colors $(k-2,k-1)$ are sent to two points whose $y$-coordinates are additive inverses of each other.   We give the coordinates of the face center images in Table~\ref{tab:face_images}. 

\begin{table}[H]
 \begin{tabular}{|c|l|l|l|l|l|}
    \hline
     & $E_1$& $E_2$& $E_3$& $E_4$& $E_5$\\\hline\hline
    $ F_{(5,1)}$ & $\overset{}{E_1(0,0)}$& $E_2(\phi+1, \pm \sqrt{\phi^3})$ &$O_3$&$E_4(\phi+1,0)$&$E_5(1,0)$\\\hline
    $F_{(1,2)}$ &$\overset{}{E_1(1,0)}$ &$E_2(\phi,0)$&$E_3(\phi+1,\pm \phi^3)$ &$O_4$& $E_5(0,0)$\\\hline
    $F_{(2,3)}$ &$\overset{}{E_1(\phi+1,0)}$ &$E_2(1,0)$&$E_3(1,0)$&$E_4(-\phi,\pm i\phi^3)$&$O_5$\\\hline
    $F_{(3,4)}$ & $O_1$ &$\overset{}{E_2(0,0)}$&$E_3(0,0)$&$E_4(0,0)$&$E_5(\phi+1,\pm
     \sqrt{\phi^3} 
    )$\\\hline
    $F_{(4,5)}$ &$\overset{}{E_1(-\phi,\pm i\phi^3)}$ &$O_2$&$E_3(-\phi,0)$&$E_4(1,0)$&$E_5(\phi,0)$\\\hline
    \end{tabular}
    \vspace*{2mm}
    
    \caption{Coordinates of face center images}\label{tab:face_images}
\end{table} 
Note that the images of points described above are not canonical. That is, a different choice of maps than that given in Equation~\ref{eq:maps} would permute the images of $F_{k-1,k}$, $F_{k,k+1}$, $F_{k+1,k+2}$, and $F_{k+2,k-2}$, and exchange the points $F_{k-2,k-1}^\pm$  and the points $W_k^\pm$, $B_k^\pm$ for  different but similarly related collection of points on $E_k$. However, the resulting Mordell-Weil group would be isomorphic to the one we describe in this section.

In our discussion of the images of points, we have been working over the complex numbers. To study the elliptic curves via a computer algebra system, it is convenient to work over a number field. Consulting Tables \ref{tab:vertex_images} and \ref{tab:face_images}, we obtain the following description of the appropriate field.

\begin{proposition}\label{prop:field_choice}
For $k=1, 3, 4$, the images in $E_k$ of the vertices and face centers of the $N=5$ Adinkra embedded in $X^5_\mathrm{alg}$ are defined over $\mathbb{Q}(\zeta, i)$, a degree 8 extension of $\mathbb{Q}$; for $k=2,5$ they are defined over $\mathbb{Q}(\zeta, i, \sqrt{\phi})$, a degree 16 extension of $\mathbb{Q}$. (Here $\zeta = e^{\frac{2\pi i}{10}}$.)
		\end{proposition}

\noindent (We will consider the question of fields of definition further in \S~\ref{SS:fieldsOfDef}.)

We are now ready to study the images of the $N=5$ Adinkra vertices and face centers that appear in {\it Jacobian divisors} on each elliptic curve. We begin with the images of the black and white vertices in each elliptic curve. 

 Using \cite{sage} and the coordinates given in Table~\ref{tab:vertex_images}, one may check that the points $B_k^+$, $B_k^-$, $W_k^+$, and $W_k^-$ have infinite order for every elliptic curve $E_k$. Because reflection across the $x$-axis corresponds to negation in the elliptic curve group law, we have the relations
$B_k^+ + B_k^- = O_k$ and $W_k^+ + W_k^- = O_k$. In each case, $B_k^+ + W_k^+$ is the image of a face center and has order 2. Specifically, we have the 
relation 
\begin{equation}
    B_k^+ + W_k^+ = F_{(k-1,k)}.
\end{equation}
We summarize this description of the relationship between images of black and white vertices in the following lemma.

\begin{lemma}\label{lem:vertices}
The images of the black and white vertices of $H^5$ in $E_k$ all have infinite order. Together, they generate an abelian group isomorphic to $\mathbb{Z} \times \mathbb{Z}/(2)$. As generators of this group, we may take either the images of a single black and a single white vertex, or the image of a vertex of either color together with the face center image $F_{(k-1,k)}$. 
 \end{lemma}

Next, we consider the images of face centers in the elliptic curve group; as we will see, these correspond to torsion points. In each elliptic curve, the face centers corresponding to one pair of adjacent edge colors are sent to the identity element for the group law (in other words, the point at infinity). We have seen that face centers corresponding to colors $(k-2,k-1)$ have two possible images in $E_{k}$. Using SageMath (\cite{sage}), one may check that each of these images has order 4, and that the remaining pairs of adjacent colors correspond to points of order 2. A more detailed study of the group law yields the following lemma. 

\begin{lemma}\label{lem:face_centers}
The images of face centers in the elliptic curve $E_k$ generate a group isomorphic to $\mathbb{Z}/(4) \times \mathbb{Z}/(2)$. A specific isomorphism may be chosen as follows. Let $\ee{k}{4}$ be a generator of $\mathbb{Z}/(4)$, let $\ee{k}{2}$ be a generator of $\mathbb{Z}/(2)$, and denote the identity by 0. In each curve, one pair of adjacent edge colors corresponds to 0, one pair of adjacent edge colors corresponds to $\pm \ee{k}{4}$, one pair of adjacent edge colors corresponds to $2 \ee{k}{4}$, and the remaining two pairs of adjacent edge colors may be assigned to $\ee{k}{2}$ and $2\ee{k}{4}+\ee{k}{2}$ in either order. Specifically, the pair of edge colors corresponding to 0 in $E_k$ is $(k-3,k-2)$, the pair of edge colors corresponding to $\pm \ee{k}{4}$ in $E_k$ is $(k-2,k-1)$, and the pair of edge colors corresponding to $2 \ee{k}{4}$ in $E_k$ is $(k-1,k)$, where we use arithmetic $\pmod{5}$.
\end{lemma}
 Combining Lemmas~\ref{lem:vertices} and~\ref{lem:face_centers}, we note:

 \begin{proposition}\label{prop:abelian_group}
The images of the vertices and face centers of $H^5$ in $E_k$ generate an abelian group isomorphic to $\mathbb{Z} \times \mathbb{Z}/(4) \times \mathbb{Z}/(2)$.
\end{proposition}

The correspondence between vertex or face center images and possible choices of generators for each elliptic curve is summarized in Table~\ref{tab:torsion_group}.

\begin{table}[ht]
    \begin{tabular}{|c|c|c|c|c|c|c|}
    \hline
    EC $\setminus$ group elements & 0  & $\pm \ee{k}{4}$ & $2 \ee{k}{4}$ & $\ee{k}{2}, 2\ee{k}{4}+\ee{k}{2}$ & $\pm \ee{k}{\infty}$ \\\hline\hline
    $\underset{}{E_1}$ & $F_{(3,4)}$ & $F_{(4,5)}^\pm$ & $F_{(5,1)}$ & $F_{(1,2)}, F_{(2,3)}$ & $W_k^\pm, B_k^\pm$\\\hline
    $\underset{}{E_2}$ & $F_{(4,5)}$ & $F_{(5,1)}^\pm$ & $F_{(1,2)}$ & $F_{(2,3)}, F_{(3,4)}$ & $W_k^\pm, B_k^\pm$\\\hline
    $\underset{}{E_3}$ & $F_{(5,1)}$ & $F_{(1,2)}^\pm$ & $F_{(2,3)}$ & $F_{(3,4)}, F_{(4,5)}$ & $W_k^\pm, B_k^\pm$ \\\hline
    $\underset{}{E_4}$ & $F_{(1,2)}$ & $F_{(2,3)}^\pm$  & $F_{(3,4)}$ & $F_{(4,5)}, F_{(5,1)}$ & $W_k^\pm, B_k^\pm$ \\\hline
    $\underset{}{E_5}$ & $F_{(2,3)}$ & $F_{(3,4)}^\pm$  & $F_{(4,5)}$  & $F_{(5,1)}, F_{(1,2)}$ & $W_k^\pm, B_k^\pm$\\\hline
    \end{tabular}
    \vspace*{.3cm}
    \caption{Possible generators of $\mathbb{Z} \times \mathbb{Z}/(4) \times \mathbb{Z}/(2)$ }\label{tab:torsion_group}
\end{table}

It is natural to ask which vertices in the $H^5$ Adinkra are mapped to the points we have designated as $B_k^+$, $B_k^-$, $W_k^+$, $W_k^-$, $F_{k-2,k-1}^+$, and $F_{k-2,k-1}^-$. The answer is given in the following lemma, and displayed visually in Figure~\ref{fig:decorations}. 


\begin{figure}[H]
\definecolor{zzttqq}{rgb}{0.6,0.2,0}
\begin{center}
        \resizebox{.34\columnwidth}{!}{%
\begin{tikzpicture}[line cap=round,line join=round,>=triangle 45,x=1cm,y=1cm,scale=1.8]
\draw [line width=1pt, color = blue] (2.5,-0.46115823141237366)-- (2.5,0.5647362809793773);
\draw [line width=1pt,color=red] (1.8969943352083505,-1.2911243263732994)-- (2.5,-0.46115823141237366);
\draw [line width=1pt,color=violet] (2.5,-0.46115823141237366)-- (1.524316338958386,-0.14413939264732145);
\draw [line width=1pt,color=orange] (3.1030056647916484,-1.2911243263733005)-- (2.5,-0.46115823141237366);
\draw [line width=1pt,color=green] (3.475683661041614,-0.1441393926473227)-- (2.5,-0.46115823141237366);
\begin{scriptsize}

\draw [fill=white] (2.5,-0.46115823141237366) circle (1.2pt);
\draw[color=black] (2.63,-0.3) node {$v$};
\draw [fill=black,color = blue] (2.5,0.5647362809793773) circle (1.2pt);
\draw[color=black] (2.483478086269523,0.7463646395605015) node {$v_1: ({-},{+},{+},{-}, {+})$};
\draw [fill=black,color=green] (3.475683661041614,-0.1441393926473227) circle (1.2pt);
\draw[color=black,anchor=south] (3.57,-0.1) node {$v_2: ({+},{-},{+},{+}, {-})$};
\draw [fill=black,color=violet] (1.524316338958386,-0.14413939264732145) circle (1.2pt);
\draw[color=black,anchor = south] (1.6,-0.14) node {$v_5: ({+},{+},{-},{+}, {-})$};
\draw [fill=black,color =red] (1.8969943352083505,-1.2911243263732994) circle (1.2pt);
\draw[color=black,anchor = north] (1.6,-1.3) node {$v_4: ({+},{-},{+},{-}, {+})$};
\draw [fill=black,color = orange] (3.1030056647916484,-1.2911243263733005) circle (1.2pt);
\draw[color=black,anchor = north] (3.57,-1.3) node {$v_3: ({-},{+},{-},{+}, {+})$};
\draw (3,-2) node {};
\end{scriptsize}
\end{tikzpicture}}
\resizebox{.64\columnwidth}{!}{%
\begin{tikzpicture}[line cap=round,line join=round,>=triangle 45,x=1cm,y=1cm,scale=1.8]
\fill[line width=1pt,fill=gray,fill opacity=0.1] (2.5,-0.46115823141237366) -- (2.5,0.5647362809793773) -- (3,1.0776835371752527) -- (3.475683661041614,-0.1441393926473227) -- cycle;
\fill[line width=1pt,fill=gray,fill opacity=0.1] (2.5,-0.46115823141237366) -- (3.475683661041614,-0.1441393926473227) -- (4.118033988749895,-0.46115823141237394) -- (3.1030056647916484,-1.2911243263733005) -- cycle;
\fill[line width=1pt,fill=gray,fill opacity=0.1] (2.5,-0.46115823141237366) -- (3.1030056647916484,-1.2911243263733005) -- (3,-2) -- (1.8969943352083505,-1.2911243263732994) -- cycle;
\fill[line width=1pt,fill=gray,fill opacity=0.1] (2.5,-0.46115823141237366) -- (1.8969943352083505,-1.2911243263732994) -- (1.1909830056250525,-1.4122147477075266) -- (1.524316338958386,-0.14413939264732145) -- cycle;
\fill[line width=1pt,fill=gray,fill opacity=0.1] (2.5,-0.46115823141237366) -- (1.524316338958386,-0.14413939264732145) -- (1.190983005625053,0.48989828488278) -- (2.5,0.5647362809793773) -- cycle;
\draw [line width=1pt,color=blue] (2.5,-0.46115823141237366)-- (2.5,0.5647362809793773);
\draw [line width=1pt,color=red] (2.5,0.5647362809793773)-- (2,1.0776835371752527);
\draw [line width=1pt,color=blue] (2,1.0776835371752527)-- (1.8969943352083505,-1.2911243263732994);
\draw [line width=1pt,color=violet] (2.5,-0.46115823141237366)-- (1.524316338958386,-0.14413939264732145);
\draw [line width=1pt,color=orange] (1.524316338958386,-0.14413939264732145)-- (0.8819660112501055,-0.4611582314123732);
\draw [line width=1pt,violet] (0.8819660112501055,-0.4611582314123732)-- (3.1030056647916484,-1.2911243263733005);
\draw [line width=1pt,color=red] (2.5,-0.46115823141237366)-- (1.8969943352083505,-1.2911243263732994);
\draw [line width=1pt,color=green] (1.8969943352083505,-1.2911243263732994)-- (2,-2);
\draw [line width=1pt,color=red] (2,-2)-- (3.475683661041614,-0.1441393926473227);
\draw [line width=1pt,color = green] (3.475683661041614,-0.1441393926473227)-- (2.5,-0.46115823141237366);
\draw [line width=1pt,color=orange] (2.5,-0.46115823141237366)-- (3.1030056647916484,-1.2911243263733005);
\draw [line width=1pt,color=blue] (3.1030056647916484,-1.2911243263733005)-- (3.809016994374947,-1.412214747707527);
\draw [line width=1pt,color=orange] (3.809016994374947,-1.412214747707527)-- (2.5,0.5647362809793773);
\draw [line width=1pt,color=violet] (3.475683661041614,-0.1441393926473227)-- (3.8090169943749475,0.4898982848827793);
\draw [line width=1pt,color=green] (3.8090169943749475,0.4898982848827793)-- (1.524316338958386,-0.14413939264732145);
\draw [line width=1pt,color=violet] (2.5,0.5647362809793773)-- (1.190983005625053,0.48989828488278);
\draw [line width=1pt,color=blue] (1.190983005625053,0.48989828488278)-- (1.524316338958386,-0.14413939264732145);
\draw [line width=1pt,color=red] (1.524316338958386,-0.14413939264732145)-- (1.1909830056250525,-1.4122147477075266);
\draw [line width=1pt,color=violet] (1.1909830056250525,-1.4122147477075266)-- (1.8969943352083505,-1.2911243263732994);
\draw [line width=1pt,color=orange] (1.8969943352083505,-1.2911243263732994)-- (3,-2);
\draw [line width=1pt,color=red] (3,-2)-- (3.1030056647916484,-1.2911243263733005);
\draw [line width=1pt,color=green] (3.1030056647916484,-1.2911243263733005)-- (4.118033988749895,-0.46115823141237394);
\draw [line width=1pt,color=orange] (4.118033988749895,-0.46115823141237394)-- (3.475683661041614,-0.1441393926473227);
\draw [line width=1pt,color=blue] (3.475683661041614,-0.1441393926473227)-- (3,1.0776835371752527);
\draw [line width=1pt,color=green] (3,1.0776835371752527)-- (2.5,0.5647362809793773);
 \begin{scriptsize}
 \draw [color = black, fill = white]
 (2,-2) circle (1.2pt);
\draw[color=black,anchor = north east] (2.45,-2.06) node {$v_{24}: (+,+,+,-,-)$};
\draw [color = black, fill = white]
(3,-2) circle (1.2pt);
\draw[color=black,anchor=north west] (2.7,-2.) node {$v_{34}: (-,-,-,-,+)$};
\draw [color = black, fill = white]
(3.809016994374947,-1.412214747707527) circle (1.2pt);
\draw[color=black,anchor = west] (3.55,-1.58) node {$v_{13}: (+,+,-,-,+)$};
\draw [color = black, fill = white]
(4.118033988749895,-0.46115823141237394) circle (1.2pt);
\draw[color=black,anchor=west] (3.87,-0.67) node {$v_{23}: (-,-,-,+,-)$};
\draw [color = black, fill = white]
(3.8090169943749475,0.4898982848827793) circle (1.2pt);
\draw[color=black,anchor = south west] (3.51,0.52) node {$v_{25}: (+,-,-,+,+)$};
 \draw [color = black, fill = white]
 (3,1.0776835371752527) circle (1.2pt);
\draw[color=black, anchor = south] (3.4,1.08) node {$v_{12}: (-,-,+,-,-)$};
\draw [color = black, fill = white]
(2,1.0776835371752527) circle (1.2pt);
\draw[color=black,anchor = south east] (2.4,1.08) node {$v_{14}: (-,-,+,+,+)$};
\draw [color = black, fill = white]
(1.190983005625053,0.48989828488278) circle (1.2pt);
\draw[color=black, anchor = south east] (1.6,0.49) node {$v_{15}: (-,+,-,-,-)$};
\draw [color = black, fill = white]
(0.8819660112501055,-0.4611582314123732) circle (1.2pt);
\draw[color=black,anchor = east] (1.25,-0.67) node {$v_{35}: (-,+,+,+,-)$};
\draw [color = black, fill = white]
(1.1909830056250525,-1.4122147477075266) circle (1.2pt);
\draw[color=black,anchor = east] (1.45,-1.56) node {$v_{45}: (+,-,-,-,-)$};
 \draw [fill=white] (2.5,-0.46115823141237366) circle (1.2pt);
 \draw[color=black] (2.63,-0.3) node {$v$};
 \draw [fill=black,color = blue] (2.5,0.5647362809793773) circle (1.2pt);
 \draw[color=black] (2.483478086269523,0.7463646395605015) node {$v_1$};
 \draw [fill=black,color=green] (3.475683661041614,-0.1441393926473227) circle (1.2pt);
 \draw[color=black] (3.67,-0.1) node {$v_2$};
 \draw [fill=black,color = violet] (1.524316338958386,-0.14413939264732145) circle (1.2pt);
\draw[color=black] (1.36,-0.1) node {$v_5$};
\draw [fill=black,color=red] (1.8969943352083505,-1.2911243263732994) circle (1.2pt);
\draw[color=black] (1.8,-1.4) node {$v_4$};
\draw [fill=black,color=orange] (3.1030056647916484,-1.2911243263733005) circle (1.2pt);
\draw[color=black] (3.25,-1.45) node {$v_3$};
\end{scriptsize}
\end{tikzpicture}}
\end{center}
\caption{The 5-tuples in left pane display sign change between the images of $v$ and $v_i$  in $E_1, \ldots, E_5$, where $v_i$ is obtained from $v$ by traveling  along color $i$. The 5-tuples in right pane display sign change between the images of $v$ and $v_{ij}$  in $E_1, \ldots, E_5$, where $v_{ij}$ is obtained from $v$ by traveling  along colors $i$ and $j$.}\label{fig:decorations}
\end{figure}
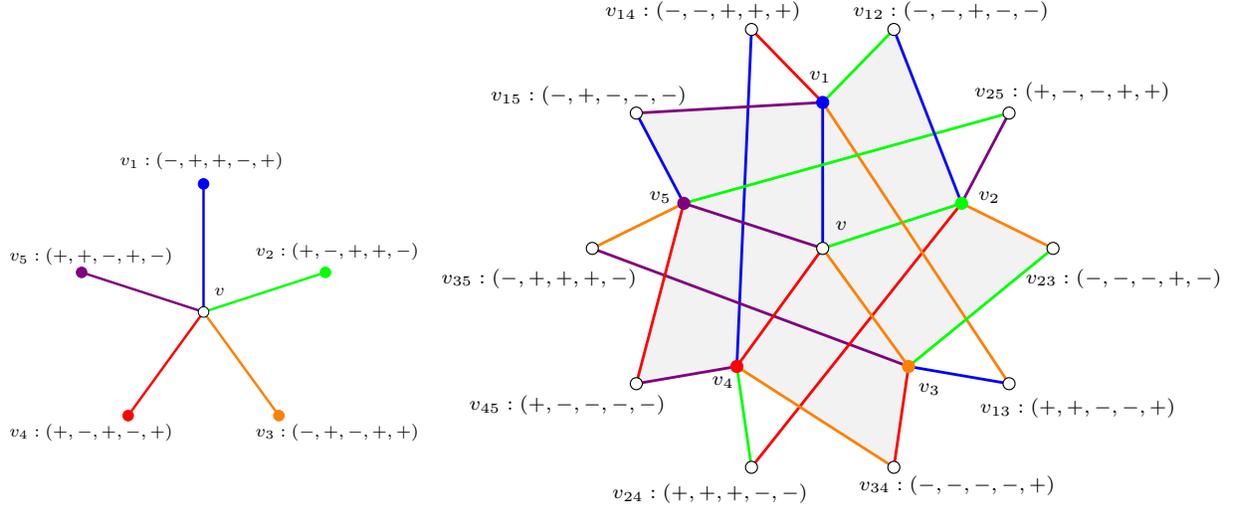

\begin{lemma}\label{lem:color_split}
Given a vertex $v \in H^5$ with image $W_k^\sigma$ (or $B_k^\sigma$) in $E_k$, the image in $E_k$ of $v_j$ (the vertex adjacent to $v$ along color $j$) is $B^\pm_k$ (or $W^\pm_k$),
where the value of $\pm$ is determined by $S_{k,v}(v_j) = (-1)^{\{k,k+2\}\cup \{j\}}$. Similarly, let ${\bf f}$ be a face center point adjacent to a certain black or white vertex ${\bf v}$ and suppose ${\bf f}'$ is the face center point adjacent to ${\bf v}_j$ (the point adjacent to ${\bf v}$ along $j$). If the image  of  ${\bf f}$ in $E_k$ is $(x,y)$, then the image of ${\bf f}'$ in $E_k$ is $(\overline{x},\pm \overline{y})$. The value of $\pm$ is determined by $S_{k,v}(v_j) = (-1)^{\{k,k+2\}\cup \{j\}}$.

\end{lemma}
   This lemma implies that under the labeling set out in Table~\ref{tab:vertex_images}, for each $k$, in a $k$-color-splitting of $H^5$ with starting vertex $(1,1,1,1,1)$, the black vertices in $H^5$ labeled  $+$ and $-$  are mapped to $B_k^+$ and $B_k^-$, respectively; the white vertices labeled  $+$ and $-$ are mapped to $W_k^+$ and $W_k^-$, respectively; and the face centers ${\bf f}_{k-2,k-1}$ labeled $+$ and $-$ are mapped to ${F}_{k-2,k-1}^+$ and ${F}_{k-2,k-1}^-$, respectively. For any vertex $u$, a $k$-color-splitting of $H^5$ with starting vertex $u$ gives the corresponding sign {\it changes} between $\nu_k(u)$ and $\nu_k(u')$ for vertices $u' \in H^5$ (see Figure~\ref{fig:decorations}).
   \begin{proof}
Assume that $\nu_k(v) = W_k^{\sigma_k}$.
Recall from Lemma \ref{lem:adj} that the sign change of the $i$th coordinates of $v$ and $v'$ which are adjacent along color $j$ is given by  $M_{i,j}$, where $M$ is the matrix displayed below.
\[
 M=  \left[\begin{array}{ccccc}
           +1 & +1 & +1 & +1 &  +1  \\
           -1 & -1 & -1 & -1 & +1  \\
           -1 & +1 & +1 & +1 &-1   \\
           -1 & -1 & +1 & +1 & -1 \\
          -1 & -1 & -1 & +1 & -1
    \end{array}\right]
    \]
Let $M^y$ be the matrix defined by $M^y_{k,j} =\frac{\nu_k(v_j)_y}{\overset{\ }{\overline{\nu_k(v)_y}}}$, where $\nu_k(v)_y$ is the $y$-coordinate of the image of $v$ under $\nu_k$. From Equation~\ref{eq:maps}, we see that this can be obtained by first considering how coordinate sign changes affect the sign of $\nu_k(v_j)_y$, and then how coordinate-wise conjugations in $v_j$ relate to  $\overline{\nu_k(v)_y}$ . The coordinate-wise sign changes in $v_j$ amount to an overall sign change of  $  M_{1,j}\cdot M_{2,j}\cdots \hat M_{k,j} \cdots \cdot M_{5,j} $ in $\nu_k(v_j)_y$ (that is, we omit the $k$th entry from the product). This gives the intermediate matrix
\[
 \left[\begin{array}{ccccc}
           +1 & -1 & +1 & -1 &  -1  \\
           -1 & +1 & -1 & +1 & -1  \\
           -1 & -1 & +1 & -1 &+1   \\
           -1 & +1 & +1 & -1 & +1 \\
          -1 & +1 & -1 & -1 & +1
    \end{array}\right]
\]
Examining the $y$-coordinates in Equation \ref{eq:maps}, we find that conjugating all coordinates of $v_j$ and conjugating the $y$-coordinate of the image point are equivalent for $k=1$, and have opposite signs for $k=2,\ldots, 5$. Thus,
\[
M^y= \left[\begin{array}{ccccc}
           +1 & -1 & +1 & -1 &  -1  \\
           +1 & -1 & +1 & -1 & +1  \\
           +1 & +1 & -1 & +1 & -1   \\
           +1 & -1 & -1 & +1 & -1 \\
          +1 &  -1 & +1 & +1 & -1
    \end{array}\right]
\]

Finally, consider the matrix $M'_{k,j} = \frac{\nu_k(v_j)_y}{(B_k^{\sigma_k})_y}$, obtained from $M^{{y}}$ by multiplying rows 1 and 4 by $-1$ (per Table \ref{tab:vertex_images}).

\[
M'=  \left[\begin{array}{ccccc}
           -1 & +1 & -1 & +1 &  +1  \\
           +1 & -1 & +1 & -1 & +1  \\
           +1 & +1 & -1 & +1 &-1   \\
           -1 & +1 & +1 & -1 & +1 \\
          +1 & -1 & +1 & +1 & -1
    \end{array}\right]
\]

Now we observe that  $\frac{\nu_k(v_j)_y}{(B_k^\sigma)_y} = M_{k,j}'=S_{k,v}(v_j)$.     

The same proof holds for the case using instead $\nu_k = B_k^{\sigma_k}$, with $M'_{k,j} = \frac{\nu_k(v_j)_y}{(W_k^{\sigma_k})_y}$. Corollary~\ref{cor:fc_geom} gives the result for face center points. 
\end{proof}

We have now completely described the images of points and face centers, establishing Theorem~\ref{thm:imagesinEC}.

\subsection{First Jacobian divisor examples}

In this section, we illustrate Theorem~\ref{thm:imagesinEC} by computing the abelian group elements associated to certain interesting Jacobian divisors. Further discussion of Jacobian divisors will be given in Section~\ref{S:Heights_Divisors_5}. 

     Set ${\bf w}_0={\bf w}(1,1,1,1,1)$, and  ${\bf b}_0={\bf b}(0,0,0,0,0)$. Let us use these points to make an explicit choice of generators for the group map referred to in Proposition~\ref{prop:abelian_group}. For each $k$ define $\ee{k}{\infty} = \nu_k({\bf w}_0)$, $\ee{k}{4} = \nu_k({\bf f}_{k-2,k-1}(1,1,1,1,1))$, and $\ee{k}{2} = \nu_k({\bf f}_{k,k+1}(1,1,1,1,1))$. (Recall that under the group law on $E_k$,  $\ee{k}{4}$ is a point of order 4, $\ee{k}{2}$ is a point of order 2, and $\ee{k}{\infty}$ is a point of infinite order.) Let ${\bf w}={\bf w}(x_1,\ldots, x_5)$, ${\bf b}={\bf b}(x_1,\ldots, x_5)$, and ${\bf f}_{k',k'+1}={\bf f}_{k',k'+1}(x_1,\ldots, x_5)$  for $k = 1, \ldots, 5$ be any white vertex, black vertex, and face centers. Their images are described below. \\
    \begin{align}\label{conv:fc}
        &{\bf w} \mapsto S_{k,{\bf w}_0}({\bf w}) \ee{k}{\infty}
        &&{\bf b} \mapsto S_{k,{\bf b}_0}({\bf b}) (\ee{k}{\infty}+2\ee{k}{4})\notag\\
        &{\bf f}_{k-2,k-1} \mapsto S_{k,{\bf w}_0}({\bf f}_{k-2,k-1}) \ee{k}{4}
        &&{\bf f}_{k-1,k} \mapsto  2\ee{k}{4}\\
        & {\bf f}_{k,k+1} \mapsto \ee{k}{2}
        & &{\bf f}_{k+1,k+2} \mapsto 2\ee{k}{4}+\ee{k}{2}\notag\\
        &{\bf f}_{k+2,k-2} \mapsto 0,&&\notag
    \end{align}
   The choice of sign $S_{k,v}$ is described in Lemma~\ref{lem:color_split}. 

We emphasize that none of the choices $\ee{k}{\infty} = \nu_k({\bf w}_0)$, $\ee{k}{4} = \nu_k({\bf f}_{k-2,k-1}(1,1,1,1,1))$,  $\ee{k}{2} = \nu_k({\bf f}_{k,k+1}(1,1,1,1,1))$ are canonical; the description of the image points resulting from a different choice for $\ee{k}{\infty}, \ee{k}{4}, \ee{k}{2}$, can be determined by applying a group isomorphism that exchanges the relevant generators (i.e., $\ee{k}{\infty} \mapsto -\ee{k}{\infty}$ or $-\ee{k}{\infty}+2\ee{k}{4}$, $\ee{k}{4} \mapsto -\ee{k}{4}$, $\ee{k}{2}\mapsto 2\ee{k}{4}+\ee{k}{2}$).

\begin{example}\label{ex:2}

In Example~\ref{ex:divisors}, we described Morse divisors obtained by considering the valise and fully extended height functions for $H^5$, as well as the height functions obtained by starting with the fully extended height and lowering either one or two divisors. We wish to compute the $k$th coordinates of the corresponding Jacobian divisors, $\nu_k(D_{\rm v})$, $\nu_k(D_{\rm fe})$, $\nu_k(D_{1})$, and $\nu_k(D_{2})$, on ${\rm Jac}(X_{\rm alg}^5)$.

We will need the images of certain specific vertices and face centers:

    \[
    \nu_k({\bf w}{(1,1,1,1,1)}) = 
        W_k^+ = \ee{k}{\infty} \text{ for } k =1,2,3,4,5,  
\qquad
       \nu_k({\bf b}{(0,0,0,0,0)}) = 
    B_k^+ =-\ee{k}{\infty}+2\ee{k}{4} \text{ for } k =1,2,3,4,5,  
    \]
    \[
    \nu_k({\bf w}{(1,0,1,0,1)}) = 
    \begin{cases}
        W_k^+= \ee{k}{\infty}  & k =1,2,3 \\
        W_k^-= -\ee{k}{\infty}  & k =4,5
    \end{cases}
, \qquad
\nu_k({\bf b}{(1,1,1,0,1)}) = 
    \begin{cases}
        B_k^+  =-\ee{k}{\infty}+2\ee{k}{4}& k =1,3,5  \\
        B_k^- =\ee{k}{\infty}+2\ee{k}{4} & k =2,4
    \end{cases}
    \]    
    \[
    \nu_k({\bf f}_{1,2}{(1,1,1,1,1)}) = 
    \begin{cases}
        E_1(1,0) = \ee{1}{2} & k = 1 \\
        E_2(\phi,0) = 2\cdot \ee{2}{4} & k =2\\
        E_3(\phi+1, \phi^3)= \ee{3}{4} & k =3\\
        E_4(\infty, \infty) = O_4   & k =4\\
        E_5(0,0)= 2\cdot \ee{5}{4} + \ee{5}{2} & k =5
    \end{cases}
    \]

We begin computing Jacobian divisor coordinates by considering the valise divisor $D_{\rm v}$. We have 

\[
D_{\rm v} = \sum_{f \text{ face center}} f - \sum_{v \text{ vertex}} v,
\]
Let $w_0 = (1,1,1,1,1)$. Then the set of all white vertices is equal to the disjoint union of (a) $\{w_0\}$, (b) those vertices that are distance 2 from $w_0$, and (c) those vertices that are distance 1 from $(0,0,0,0,0)$. By combining all labels from Figure \ref{fig:decorations}, we see that 
\[
\sum_{w \text{ white vertex}} \nu_k(w) = W_k^+ +2W_k^- + \overline{B_k^+}.
\]
Similarly, 
\[
\sum_{b \text{ black vertex}} \nu_k(b) = B_k^+ +2B_k^- + \overline{W_k^+},
\]
and thus
\[
\sum_{v \text{ vertex}} \nu_k(w) = O_k.
\]
Next, we consider face centers. For each $j \in [5]$, there are 8 $j,j+1$ colored face centers in $H^5$. On $E_k$, $\nu_k(\sum_{f_{j,j+1}} f_{j,j+1}) = O_k$ trivially for all $j \neq k-2 \mod 5$  (cf. Table \ref{tab:torsion_group}). In the case that $j \equiv k-2 \mod 5$, set $w = (1,1,1,1,1)$. Then Figure \ref{fig:decorations} shows that  $\nu_{k}(f_{k-2,k-1}(w)) = -e_4 $ for  $w_k$, $v_{k+2}, w_{k,k+1}, w_{k+1,k+2}$ and $\nu_{k}(f_{k-2,k-1}(w)) = e_4 $ for the remaining four cases. Thus 
\[
 \sum_{f \text{ face center}} \nu_k(f)=O_k,
\qquad \text{and}\qquad  
\nu_k(D_{\rm v}) = O_k.
\]

Next, consider the fully extended divisor $D_{\rm fe}$.
We have 
\[
\nu_k(D_{\rm fe}) = \sum_{v \in \mathcal{V}} v - (1,1,1,1,1) - (0,0,0,0,0),
\]
where 
\[
\mathcal{V} = \left\{
\begin{array}{c}(1,0,1,0,1),(0,1,0,1,1),(1,0,1,1,0),(0,1,1,0,1),(1,1,0,1,0),\\(0,1,0,1,0), (1,0,1,0, 0), (0,1,0, 0,1), (1,0, 0,1,0), (0, 0,1,0,1)\end{array}\right\}.
\]

From Lemma~\ref{lem:color_split}, we find the values in the following table
\begin{center}
    \begin{tabular}{|c| c| c| c| c| c|| c| c| c| c| c| c|}
    \hline
    & $E_1$ &  $E_2$ &  $E_3$ &  $E_4$ &  $E_5$ & 
    & $E_1$ &  $E_2$ &  $E_3$ &  $E_4$ &  $E_5$ \\\hline
    $\nu_k({\bf w}(1,0,1,0,1)) $&$ W_1^+$&$W_2^+$ & $W_3^+$ & $W_4^-$ & $W_5^-$ & $\nu_k({\bf b}(0,1,0,1,0)) $ & $B_1^+$ & $B_2^+$ & $B_3^+$ & $B_4^-$ & $B_5^-$\\\hline
   $\nu_k({\bf w}(1,1,0,1,0)) $ & $W_1^-$ & $W_2^+$ & $W_3^+$ & $W_4^+$ & $W_5^-$ &
   $\nu_k({\bf b}(0,0,1,0,1)) $ & $B_1^-$ & $B_2^+$ & $B_3^+$ & $B_4^+$ & $B_5^-$\\\hline
    $\nu_k({\bf w}(0,1,1,0,1)) $ & $W_1^-$ & $W_2^-$ & $W_3^+$ & $W_4^+$ & $W_5^+$ &
   $\nu_k({\bf b}(1,0,0,1,0)) $ & $B_1^-$ & $B_2^-$ & $B_3^+$ & $B_4^+$ & $B_5^+$\\\hline
    $\nu_k({\bf w}(1,0,1,1,0)) $ & $W_1^+$ & $W_2^-$ & $W_3^-$ & $W_4^+$ & $W_5^+$ & 
   $\nu_k({\bf b}(0,1,0,0,1)) $ & $B_1^+$ & $B_2^-$ & $B_3^-$ & $B_4^+$ & $B_5^+$\\\hline
    $\nu_k({\bf w}(0,1,0,1,1)) $ & $W_1^+$ & $W_2^+$ & $W_3^-$ & $W_4^-$ & $W_5^+$ &
   $\nu_k({\bf b}(1,0,1,0,0))$ & $B_1^+$ & $B_2^+$ & $B_3^-$ & $B_4^-$ & $B_5^+$\\\hline
    \end{tabular}
\end{center}
\vspace*{2mm}

Note that for each $k=1,\ldots, 5$, we now have 
\[
\sum_{v \in \mathcal{V}} \nu_k(v) = W_k^+ + B_k^+.
\]
Thus, 
$\nu_k(D_{\rm fe} )= 
O_k\in E_k$.
\vspace{.2cm}

\noindent The geometric divisor $D_1$ associated to the height function $h_1$ in Example \ref{ex:divisors}{\bf (c)} differs from $D_{\rm fe}$ by 
\[
D_{1}-D_{\rm fe}=\sum_{v \in {\mathcal{V}''}} v -\sum_{v \in {\mathcal{V}'}} v
\]
for 
\begin{align*}
\mathcal{V}'' &= \{{\bf f}_{12}(1,1,1,1,1), {\bf f}_{23}(1,1,1,1,1),{\bf f}_{34}(1,1,1,1,1),{\bf f}_{45}(1,1,1,1,1),{\bf f}_{51}(1,1,1,1,1),
\},\\[2mm]
\mathcal{V}' &= \{{\bf b}(0,1,1,1,1),{\bf b}(1,0,1,1,1),{\bf b}(1,1,0,1,1),{\bf b}(1,1,1,0,1),{\bf b}(1,1,1,1,0),
\},\\
\textbf{}\end{align*}
Using the description in Equation~\ref{conv:fc}, we find that

\begin{center}
    \begin{tabular}{|c|c|c|c|c|c|}
    \hline
    & $E_1$& $E_2$& $E_3$& $E_4$& $E_5$\\\hline
    $\nu_k({\bf f}_{12}(1,1,1,1,1))$&$\ee{1}{2}$&$2\cdot \ee{2}{4}$&$\ee{3}{4}$&$O_4$&$2\cdot \ee{5}{4}+\ee{5}{2}$ \\\hline
    $\nu_k({\bf f}_{23}(1,1,1,1,1))$&$2\cdot \ee{1}{4}+\ee{1}{2}$&$\ee{2}{2}$&$2\cdot \ee{3}{4}$&$\ee{4}{4}$&$O_5$ \\\hline
    $\nu_k({\bf f}_{34}(1,1,1,1,1))$&$O_1$&$2\cdot \ee{2}{4}+\ee{2}{2}$&$\ee{3}{2}$&$2\cdot \ee{4}{4}$&$\ee{5}{4}$ \\\hline
    $\nu_k({\bf f}_{45}(1,1,1,1,1))$&$\ee{1}{4}$&$O_2$&$2\cdot \ee{3}{4}+\ee{3}{2}$&$\ee{4}{2}$&$2\cdot \ee{5}{4}$ \\\hline
    $\nu_k({\bf f}_{51}(1,1,1,1,1))$&$2\cdot \ee{1}{4}$&$\ee{2}{4}$&$O_3$&$2\cdot \ee{4}{4}+\ee{4}{2}$&$\ee{5}{2}$ \\\hline\hline
    $\nu_k({\bf b}(0,1,1,1,1))$&$B_1^-$&$B_2^+$&$B_3^+$&$B_4^-$&$B_5^+$\\\hline
    $\nu_k({\bf b}(1,0,1,1,1))$&$B_1^+$&$B_2^-$&$B_3^+$&$B_4^+$&$B_5^-$\\\hline
    $\nu_k({\bf b}(1,1,0,1,1))$&$B_1^-$&$B_2^+$&$B_3^-$&$B_4^+$&$B_5^+$\\\hline
    $\nu_k({\bf b}(1,1,1,0,1))$&$B_1^+$&$B_2^-$&$B_3^+$&$B_4^-$&$B_5^+$\\\hline
    $\nu_k({\bf b}(1,1,1,1,0)) $&$B_1^+$&$B_2^+$&$B_3^-$&$B_4^+$&$B_5^-$\\\hline
    \end{tabular}
\end{center}
\vspace*{2mm}

Note that for each $k=1,\ldots, 5$, we now have 
\[
\sum_{v \in \mathcal{V}''}\nu_k(v) -\sum_{v \in {\mathcal{V}'}} \nu_k(v) =  \ee{k}{4} - B_k^+
= \ee{k}{4} +\ee{k}{\infty} +2\ee{k}{4} .
\]
Thus, 
\begin{align*}
\nu_k(D_{1}) &= 
\ee{k}{\infty} -\ee{k}{4}
\end{align*}

As discussed in Example \ref{ex:divisors}{\bf (d)}, the geometric divisor $D_2$ associated to the height $h_2$  differs from $D_{1}$ by 

 \begin{multline*}
 D_2 - D_1 = (1,1,1,1,1)+\left(0,\frac12, \frac12,1,1\right)+\left(0,1,\frac12, \frac12,1\right)+\left(0,1,1,\frac12, \frac12\right)\\
 -(0,1,0,1,1)-(0,1,1,0,1)-\left(\frac 12, \frac 12, 1,1,1 \right)-\left(\frac 12, 1, 1,1,\frac 12\right)
 \end{multline*}

\begin{center}
    \begin{tabular}{|c|c|c|c|c|c|}
    \hline
    & $E_1$& $E_2$& $E_3$& $E_4$& $E_5$\\\hline
    $-\nu_k({\bf f}_{12}(1,1,1,1,1))$&$-\ee{1}{2}$&$-2\cdot \ee{2}{4}$&$-\ee{3}{4}$&$O_4$&$-2\cdot \ee{5}{4}-\ee{5}{2}$ \\\hline
    $\nu_k({\bf f}_{23}(0,1,1,1,1))$&$2\cdot \ee{1}{4}+\ee{1}{2}$&$\ee{2}{2}$&$2\cdot \ee{3}{4}$&$-\ee{4}{4}$&$O_5$ \\\hline
    $\nu_k({\bf f}_{34}(0,1,1,1,1))$&$O_1$&$2\cdot \ee{2}{4}+\ee{2}{2}$&$\ee{3}{2}$&$2\cdot \ee{4}{4}$&$\ee{5}{4}$ \\\hline
    $\nu_k({\bf f}_{45}(0,1,1,1,1))$&$-\ee{1}{4}$&$O_2$&$2\cdot \ee{3}{4}+\ee{3}{2}$&$\ee{4}{2}$&$2\cdot \ee{5}{4}$ \\\hline
    $-\nu_k({\bf f}_{51}(1,1,1,1,1))$&$-2\cdot \ee{1}{4}$&$-\ee{2}{4}$&$O_3$&$-2\cdot \ee{4}{4}-\ee{4}{2}$&$-\ee{5}{2}$ \\\hline
    $\nu_k({\bf w}(1,1,1,1,1))$&$W_1^+ $&$W_2^+$&$W_3^+$&$W_4^+$&$W_5^+$\\\hline
    $-\nu_k({\bf w}(0,1,0,1,1))$&$-W_1^+$&$-W_2^+$&$-W_3^-$&$-W_4^-$&$-W_5^+$\\\hline
    $-\nu_k({\bf w}(0,1,1,0,1))$&$-W_1^-$&$-W_2^-$&$-W_3^+$&$-W_4^+$&$-W_5^+$\\\hline\hline
    $\nu_k(D_2-D_1)$&$\ee{1}{\infty}-\ee{1}{4}$&$\ee{2}{\infty}-\ee{2}{4}$&$\ee{3}{\infty}-\ee{3}{4}$&$\ee{4}{\infty}-\ee{4}{4}$&$-\ee{5}{\infty}+\ee{5}{4}$\\\hline
    \end{tabular}
\end{center}
\vspace*{2mm}

Thus, the $k$th coordinate of the Jacobian divisor is
\[
\nu_k(D_{2}) =\begin{cases}
 2\ee{k}{\infty} +2\ee{k}{4}& k \neq 5\\
 O_4 & k = 5.
 \end{cases}
\]

\end{example}

\FloatBarrier

\subsection{Fields of definition}\label{SS:fieldsOfDef}

If we view each elliptic curve $E_k$ as defined over a number field $K_k$, the question naturally arises of how the abelian group described in Proposition~\ref{prop:abelian_group} compares to the Mordell-Weil group $MW(E_k/K_k)$. Of course, $E_k$ itself can be defined over more than one number field; one might also ask whether the fields specified in Proposition~\ref{prop:field_choice} are in some sense minimal. One may use SageMath (\cite{sage}) to check that the rank of each of the elliptic curves $E_k$ over $\mathbb{Q}(\phi)$ is 0. Thus, we must extend $\mathbb{Q}(\phi)$ if we wish points of infinite order to exist. We characterize the minimal appropriate extension in the following proposition.

\begin{proposition}\label{prop:field_of_def}
    Let $\widehat{K_k}$ be a number field such that $MW(E_k/\widehat{K_k})$ contains a point of infinite order and either $\mathbb{Q}(\phi) \subseteq \widehat{K_k} \subseteq \mathbb{Q}(\zeta, i)$ or $\mathbb{Q}(\zeta, i) \subseteq \widehat{K_k} \subseteq \mathbb{Q}(\zeta, i, \sqrt{\phi})$. Then $[\widehat{K_k}:\mathbb{Q}] \geq 8$ for $k=1,4$, $[\widehat{K_k}:\mathbb{Q}] = 16$ for $k= 2,5$, and $[\widehat{K_k}:\mathbb{Q}] \geq 4$ for $k=3$.
 \end{proposition} 

\begin{proof}
The Galois group $\mathrm{Gal}(\mathbb{Q}(\zeta, i):\mathbb{Q})$ is isomorphic to $\mathbb{Z}/(4) \times \mathbb{Z}/(2)$; the Galois group $\mathrm{Gal}(\mathbb{Q}(\zeta, i):\mathbb{Q}(\phi))$ is isomorphic to $\mathbb{Z}/(2) \times \mathbb{Z}/(2)$. Over $\mathbb{Q}$, the minimal polynomial of $\zeta = e^{2\pi i/10}$ is $x^4 - x^3 + x^2 - x + 1$. This polynomial factors as $(x^2 - \phi x + 1)  (x^2 + (\phi - 1)x + 1)$ over $\mathbb{Q}(\phi)$. By the fundamental theorem of Galois theory, there are five fields containing $\mathbb{Q}(\phi)$ and contained in $\mathbb{Q}(\zeta, i)$, including these fields themselves; we write the intermediate fields as $\mathbb{Q}(\phi, i)$, $\mathbb{Q}(\phi, i\sqrt{3-\phi})$, and $\mathbb{Q}(\phi, i\sqrt{\phi+2})$. We use SageMath (\cite{sage}) to compute the rank of $E_k$ over these fields for each $k$. The result of this computation is given in Table~\ref{tab:ec_rank}. We see that $E_1=E_4$ has a point of infinite order for $\mathbb{Q}(\zeta, i)$ but not the intermediate fields, while $E_3$ has points of infinite order in two different quadratic extensions of $\mathbb{Q}(\phi)$. We have already demonstrated that $E_2=E_5$ has points of infinite order over $\mathbb{Q}(\zeta, i, \sqrt{\phi})$.
\end{proof}

\begin{table}[ht]
    \begin{tabular}{|c|c|c|c|}
    \hline
     & rank of $E_1 = E_4$ & rank of $E_2 = E_5$ & rank of $E_3$ \\
     \hline
    $\mathbb{Q}(\phi)$ & 0 & 0& 0 \\
       \hline
    $\mathbb{Q}(\phi, i)$ & 0 & 0 & 0\\ \hline
    $\mathbb{Q}(\phi, i\sqrt{3-\phi})$ & 0 & 0 & 1\\ \hline
    $\mathbb{Q}(\phi, i\sqrt{\phi+2})$ &  0&  0& 1\\ \hline
    $\mathbb{Q}(\zeta, i)$ &  1&  0& 1\\
    \hline
    \end{tabular}
    \caption{Ranks of $E_k$ over extensions of $\mathbb{Q}(\phi)$}\label{tab:ec_rank}
    \end{table}

Similarly, we may use SageMath to compute the torsion part of the Mordell-Weil group of the elliptic curves $E_k$ over different field extensions. The results of this computation are shown in Table~\ref{tab:torsion_over_extension}. 

\begin{table}[ht]
    \begin{tabular}{|c|c|c|c|}
    \hline
     &  $MW_\mathrm{tor}(E_1 = E_4)$ &  $MW_\mathrm{tor}(E_2 = E_5)$ &  $MW_\mathrm{tor}(E_3)$ \\
     \hline
    $\mathbb{Q}(\phi)$ & $\mathbb{Z}/(2) \times \mathbb{Z}/(2)$  & $\mathbb{Z}/(2) \times \mathbb{Z}/(2)$ & $\mathbb{Z}/(4) \times \mathbb{Z}/(2)$ \\
       \hline
    $\mathbb{Q}(\phi, i)$ & $\mathbb{Z}/(4) \times \mathbb{Z}/(2)$ & $\mathbb{Z}/(2) \times \mathbb{Z}/(2)$ & $\mathbb{Z}/(4) \times \mathbb{Z}/(2)$\\ \hline
    $\mathbb{Q}(\phi, i\sqrt{3-\phi})$ & $\mathbb{Z}/(2) \times \mathbb{Z}/(2)$ & $\mathbb{Z}/(2) \times \mathbb{Z}/(2)$ & $\mathbb{Z}/(4) \times \mathbb{Z}/(2)$ \\ \hline
    $\mathbb{Q}(\phi, i\sqrt{\phi+2})$ & $\mathbb{Z}/(2) \times \mathbb{Z}/(2)$ & $\mathbb{Z}/(2) \times \mathbb{Z}/(2)$ & $\mathbb{Z}/(4) \times \mathbb{Z}/(2)$  \\ \hline
    $\mathbb{Q}(\zeta, i)$ & $\mathbb{Z}/(4) \times \mathbb{Z}/(2)$ & $\mathbb{Z}/(2) \times \mathbb{Z}/(2)$  &  $\mathbb{Z}/(4) \times \mathbb{Z}/(2)$ \\
    \hline
    $\mathbb{Q}(\zeta, i, \sqrt{\phi})$ & $\mathbb{Z}/(8) \times \mathbb{Z}/(4)$ & $\mathbb{Z}/(8) \times \mathbb{Z}/(4)$  &  $\mathbb{Z}/(8) \times \mathbb{Z}/(4)$ \\
    \hline
    \end{tabular}
    \caption{Torsion groups over extensions of $\mathbb{Q}(\phi)$}\label{tab:torsion_over_extension}
    \end{table}

Proposition~\ref{prop:field_of_def} and Tables~\ref{tab:ec_rank} and \ref{tab:torsion_over_extension} show that the difference between the field of definition for $E_2=E_5$ and the other elliptic curves specified in Proposition~\ref{prop:field_choice} is in some sense inevitable: we genuinely need the square root of $\phi$ to describe points of order more than 2 on $E_2 = E_5$.

\subsection{Heights and Morse Divisors on ${\rm Jac}(X_{\rm alg}^5)$} \label{S:Heights_Divisors_5} We can apply Theorem~\ref{thm:imagesinEC} to characterize the relationship between height functions and Jacobian divisors in the $N = 5$ case. Our goal is to describe Jacobian divisors obtained as images of height functions and thereby prove our  \hyperref[thm:main]{Main Theorem}.

A step towards the classification of Jacobian divisors determined by height functions is given in the following proposition, which significantly restricts the structure of Jacobian divisors for $N = 5$.

\begin{proposition}\label{prop:raiselower}
Let $\nu_k(D_h)$ 
and $\nu_k(D_{h'})$ 
be the $E_k$ coordinates of the Jacobian divisors associated to heights $h$ and $h'$ on $H^5$, where $h'$ is obtained from $h$ by raising or lowering a single vertex $v$. Then $\nu_k(D_{h'}) - \nu_k(D_{h})=\pm(\ee{k}{\infty} - \ee{k}{4})$. 
\end{proposition}
\begin{proof}
Assume that $v$  is a local min in $h$ with $h(v) = 0$, and that $h'$ is the height obtained by raising $v$ ($h'(v) = 2$). Note that then $h(v_j) = h'(v_j) = 1$ for each $v_j$ adjacent to $v$ via color $j$. 

The only points whose contribution might differ in  $D_h$ and $D_{h'}$ are $v, v_1, \ldots, v_5, f_1 = f_{12}(v), \ldots, f_5 = f_{51}(v)$. Since $v$ is a min under $h$ and a max under $h'$, then $\kappa_v = \kappa'_v = 1$, where $\kappa_v$ and $\kappa'_v$ are the coefficients of $v$ in $D_h$ and $D_{h'}$, respectively (See Definition~\ref{def:dMD}). 

 Let $\Delta_{v_j} = \kappa'_{v_j} -\kappa_{v_j}$. Then,  
 \[
 \nu_k(D_{h'}-D_h) = \sum_{j=1}^5 (\Delta_{v_j} \nu_k(v_j)+ \Delta_{f_j} \nu_k(f_j)).
 \]

For $\mathbf 1 = (1,1,1,1,1)$, Lemmas \ref{lem:color_split} and \ref{lem:vertices} imply  that
$$
\nu_k(v_j) = 2e_4^k-S_{k,v}(v_j)\nu_k(v) = \begin{cases}
    2e_4^k-S_{k,v}(v_j)S_{ k, \mathbf{1}}(v)e_\infty^k & v \text{ white}\\
    S_{k,v}(v_j)S_{ k, \mathbf{1}}(v)e_\infty^k & v \text{ black}
\end{cases}
$$
 Also,  since $\Delta_{f_j} = \pm 1$ for each $j$, 
  \begin{align*}
 \textstyle\sum_{j=1}^5  \Delta_{f_j} \nu_k(f_{j} )&= \Delta_{f_{k-2}} \nu_k(f_{k-2})+\Delta_{f_{k-1}} \nu_k(f_{k-1})
     +\Delta_{f_k} \nu_k(f_{k})+\Delta_{f_{k+1}} \nu_k(f_{k+1})+\Delta_{f_{k+2}} \nu_k(f_{k+2}) \\
  &=\Delta_{f_{k-2}} S_{ k, \mathbf{1}}(v)  \ee{k}{4}+2\Delta_{f_{k-1}}\ee{k}{4}
 +\Delta_{f_k} \ee{k}{2}+\Delta_{f_{k+1}}(2\ee{k}{4}+\ee{k}{2})\\
 &=\Delta_{f_{k-2}} S_{ k, \mathbf{1}}(v) \ee{k}{4}.
 \end{align*}
 Thus, 
 \begin{equation*}
     \sigma \nu_k(D_{h'}-D_h) =\begin{cases}
         2\ee{k}{4} -S_{ k, \mathbf{1}}(v)e_\infty^k(\Delta_{v_{k-2}}+\Delta_{v_{k-1}} -\Delta_{v_{k}}+\Delta_{v_{k+1}} -\Delta_{v_{k+2} })+ S_{ k, \mathbf{1}}(v)\Delta_{f_{k-2}}\ee{k}{4}& v \text{ white}\\
          S_{ k, \mathbf{1}}(v)e_\infty^k(\Delta_{v_{k-2}}+\Delta_{v_{k-1}} -\Delta_{v_{k}}+\Delta_{v_{k+1}} -\Delta_{v_{k+2}} )+ S_{ k, \mathbf{1}}(v)\Delta_{f_{k-2}}\ee{k}{4}& v \text{ black}\\
     \end{cases}
 \end{equation*}
 
 We must show that  $\{\Delta_{v_{k-2}}+\Delta_{v_{k-1} }-\Delta_{v_{k}}+\Delta_{v_{k+1}} -\Delta_{v_{k+2}},\Delta_{f_{k-2}} \} = \{1,-1\}$. \\
 
 It is straightforward to verify that 
\begin{equation}\label{eq:Delta_i}
\Delta_{v_j} =\begin{cases}
    -1 & h(v_{j-1,j}) = h(v_{j,j+1})=2\\
    1 & h(v_{j-1,j}) = h(v_{j,j+1})=0\\
    0 & \text{otherwise  },
\end{cases} \quad\text{and} \quad
\Delta_{f_j} =\begin{cases}
    1 &  h(v_{j,j+1})=2\\
    -1 &  h(v_{j,j+1})=0\\
\end{cases}
\end{equation}
where $v_{j,j+1}$ is the vertex obtained from $v$ by traveling on an edge of color $j$ and an edge of color $j+1$. 
Consider, for example, the situation where \[
(h(v_{1,2}), h(v_{2,3}), h(v_{3,4}), h(v_{4,5}), h(v_{5,1})) = (2,2,2,2,0)
\]
Then for $k = 3$, $\Delta_{f_{k-2}} = \Delta_{f_1} = 1$, and $(\Delta_{v_{k-2}},\Delta_{k-1} ,\Delta_{k},\Delta_{k+1},\Delta_{k+2}) = (0,-1,-1,-1,0)$, so
\[
(\Delta_{k-2}+\Delta_{k-1} -\Delta_{k}+\Delta_{k+1} -\Delta_{k+2}, \Delta^f_{k-2}) = (-1,1).
\]
For $k = 2$, $\Delta_{f_{k-2}} = \Delta_{f_5} = -1$, and $(\Delta_{v_{k-2}},\Delta_{k-1} ,\Delta_{k},\Delta_{k+1},\Delta_{k+2}) = (0, 0,-1,-1,-1)$, so
\[
(\Delta_{k-2}+\Delta_{k-1} -\Delta_{k}+\Delta_{k+1} -\Delta_{k+2}, \Delta^f_{k-2}) = (1,-1).
\]

The following table gives all   $\Delta_{v_j}$ and  $\Delta_{f_j}$ possibilities.
\[
    \begin{array}{|c|c|c|}
    \hline
    (h(v_{j,j+1}))_j  &(\Delta_{v_j})_j & (\Delta_{f_j})_j \\\hline \hline
    (2,2,2,2,2) & (-1, -1, -1, -1, -1)&   (1,\ 1,\ 1,\ 1,\ 1)  \\ \hline
    (0,0,0,0,0) & (1,\ 1,\ 1,\ 1,\ 1)  &    (-1, -1, -1, -1, -1)\\ \hline
    (2,2,2,0,0) & (0,- 1,- 1,\ 0, \ 1)   &  (1,\ 1,\ 1,-1,-1)\\ \hline
    (0,0,0,2,2) & (0,\ 1,\ 1,\ 0, -1) &  (-1,- 1,-1,\ 1,\ 1)\\ \hline
    (2,2,2,2,0) & (0,- 1,- 1,- 1,\ 0)   &  (1,\ 1,\ 1,\ 1,-1)\\ \hline
    (0,0,0,0,2) & (0,\ 1,\ 1,\ 1,\ 0)   &  (-1,- 1,- 1,- 1,\ 1)\\ \hline
    (2,0,2,0,2) & (-1,\ 0,\ 0,\ 0,\ 0)   &  (1,-1,\ 1,-1,\ 1)\\ \hline
    (0,2,0,2,0) &  (1,\ 0,\ 0,\ 0,\ 0)  &  (-1,\ 1,- 1,\ 1,- 1)\\ \hline
    \end{array}
\]
It is straightforward to check that in all cases $\{\Delta_{v_{k-2}}+\Delta_{v_{k-1} }-\Delta_{v_{k}}+\Delta_{v_{k+1}} -\Delta_{v_{k+2}},\Delta_{f_{k-2}} \} = \{1,-1\}$.
 \end{proof}

Proposition~\ref{prop:raiselower} says that raising or lowering a vertex changes the $k$th coordinate of $\eta(D_h)$ by  $\pm(\ee{k}{\infty} - \ee{k}{4})$. Recall from  Example~\ref{ex:2},  that $\ee{k}{\infty} - \ee{k}{4}$ is exactly $\nu_k(D_1)$. Thus, integer multiples of $\nu_k(D_1)$ are the only possibilities for coordinates of the image of any divisor on ${\rm Jac}(X)$.  

\begin{corollary}
\[
\langle\, \{\nu_k(D_h) \}\mid h \text{ a height on }H^5\,\rangle \cong \langle \nu_k(D_1) \rangle\cong\mathbb{Z} \quad \text{for any }k \in [5].
\]
\end{corollary}

Though the group generated by height functions on $H^5$ is infinite, there are only a finite number of height functions up to overall shift. A natural question is which multiples $m\cdot \nu_k(D_1)$ of $\nu_k(D_1)$ occur as some $\nu_k(D_h)$. A preliminary estimate may be obtained using the  following corollary of the following Proposition.

\begin{proposition}\label{prop:Gamma_paths}
    A path from any height $h$ to a valise height $h_{\rm v}$ on $\Gamma_N$ can be attained in $m$ vertex lowering steps, where $m \leq (N-1)2^{N-2}$. 
\end{proposition}
\begin{proof}
Recall that $\Gamma_N$ is the digraph whose vertices are heights on $H^N$, where adjacency is determined by vertex lowering (see Remark~\ref{rem:Gamma}). We show that the proposition is true for the fully extended height $h_{\rm fe}$. Any other type of height requires strictly less steps to reach the valise. Note that $h_{\rm fe}$ can be defined by assigning to any vertex $v = (x_1, \ldots, x_N) \in H^N$ the height $h_{\rm fe}(v) = \sum_i x_i$. If each such vertex is lowered $\lfloor h(v)/2\rfloor$ times, all vertex heights will be 0 or 1, according to whether the vertex is black or white. The number of steps in this process is
\[
\sum_{k=0}^{N} \lfloor k/2\rfloor {N \choose k}.
\]
Using the identities 
\[
2 \left\lfloor \frac k 2 \right\rfloor = \begin{cases}
    k & k \text{ even}\\
    k-1 & k \text{ odd,}
\end{cases}
\qquad 
\sum_i i {n \choose i } = n2^{n-1}, \qquad 
\sum_{i\text{ odd}} {n \choose i } =2^{n-1}, 
\]
It is straightforward to verify that 
\[
\sum_{k=0}^{N} \lfloor k/2\rfloor {N \choose k}  = (N-1)2^{N-2}.
\]
\end{proof}

\begin{corollary}\label{cor:restricted_divisors}
For $h$ a height function on $H^5$,  $\nu_k(D_h) = a(\ee{k}{\infty} - \ee{k}{4})$ for some integer $a$ satisfying $0\leq |a| \leq 16$.
\end{corollary}
\begin{proof}
Let $h$ be a height function on $H^5$. Then, on $\Gamma_5$, $h$ lies on a minimal path between the height corresponding to the fully extended hypercube and the valise hypercube. Proposition \ref{prop:Gamma_paths} gives the result. 
\end{proof}

\begin{table}[h]
    \centering
    \begin{tabular}{|c|c|}
    \hline
        $\nu_3(D_h)$  & Frequency \\\hline\hline
       $-8(\ee{\infty}{3} - \ee{4}{3})$  & 24 \\\hline
       $-7(\ee{\infty}{3} - \ee{4}{3})$  & 128 \\\hline
       $-6(\ee{\infty}{3} - \ee{4}{3})$  & 704 \\\hline
       $-5(\ee{\infty}{3} - \ee{4}{3})$  & 2752 \\\hline
       $-4(\ee{\infty}{3} - \ee{4}{3})$  & 9048 \\\hline
       $-3(\ee{\infty}{3} - \ee{4}{3})$  & 23392 \\\hline
       $-2(\ee{\infty}{3} - \ee{4}{3})$ & 47200 \\\hline
      $-1(\ee{\infty}{3} - \ee{4}{3})$  & 72384  \\\hline
       $O_1$  & 83830  \\\hline
       $1(\ee{\infty}{3} - \ee{4}{3})$  & 72384 \\\hline
       $2(\ee{\infty}{3} - \ee{4}{3})$  & 47200 \\\hline
      $3(\ee{\infty}{3} - \ee{4}{3})$ & 23392 \\\hline
       $4(\ee{\infty}{3} - \ee{4}{3})$  & 9048 \\\hline
       $5(\ee{\infty}{3} - \ee{4}{3})$  & 2752 \\\hline
       $6(\ee{\infty}{3} - \ee{4}{3})$  & 704 \\\hline
       $7(\ee{\infty}{3} - \ee{4}{3})$  & 128 \\\hline
       $8(\ee{\infty}{3} - \ee{4}{3})$  & 24 \\\hline
    \end{tabular} 
    \vspace*{2mm}
    
    \caption{A frequency table describing images in $E_3$ of the discrete Morse divisors associated with each of the $395\, 094$ heights on $H^5$. }\label{tab:heightsdivisors_sample}
\end{table}

\begin{figure}[h]
    \includegraphics[width=.8\textwidth]{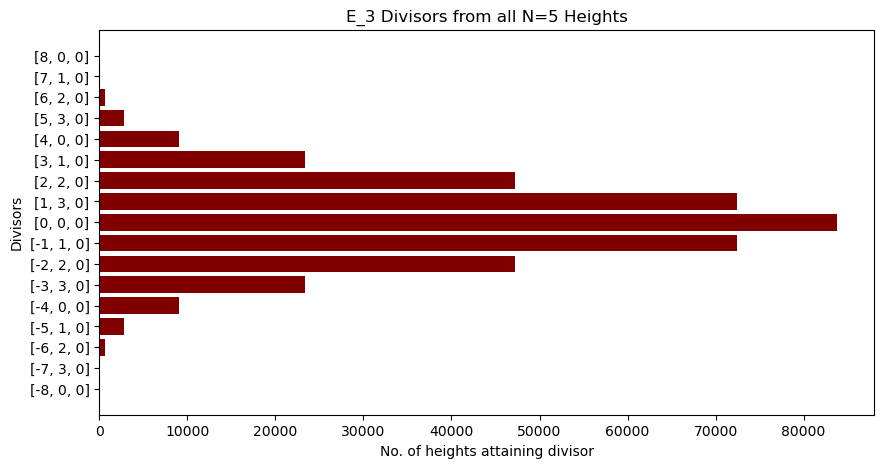}
    \caption{A frequency chart describing images in $E_3$ of the discrete Morse divisors associated with each of the $395\, 094$ heights on $H^5$. }\label{fig:heightsdivisors_sample}
\end{figure}

In fact,  we can say even more. Using SageMath, we implemented a version of Zhang's algorithm for enumerating the $N=5$ hypercube Adinkra heights (see Section~\ref{SS:CountingHeights} and \cite{ourcode}), and then performed the combinatorial algorithm laid out in Definition~\ref{def:dMD}, Equation~\ref{conv:fc},  and Lemma~\ref{lem:color_split} to find the images of each of the $395\,094$ heights as divisors in $E_3$.    

\begin{proposition}
    The images $\nu_3(D_h)$ of all $N=5$ hypercube Adinkra heights $h$ on $E_3$ form the set $\{a(\ee{\infty}{3} - \ee{4}{3}) \mid a \in \mathbb{Z}, -8 \leq a \leq 8\}$.
\end{proposition}

\noindent We list the frequency with which each possible divisor on $E_3$ appears in Table~\ref{tab:heightsdivisors_sample} and illustrate the frequencies with a chart in Figure~\ref{fig:heightsdivisors_sample}. 

One might naturally wonder whether the bound we have observed on the size of divisors obtained as height images and their frequency distribution in $E_3$ holds for every $E_k$. Indeed, it does, as advertised in our \hyperref[thm:main]{Main Theorem}. To show the pattern holds, we will need the following definition and proposition describing the images of heights on $H^5$ under what we call rainbow rotation:
\begin{definition}\label{def:rainbowequiv}
    Let $u$ be a vertex in they hypercube adinkra $H^N$.  For  $v \in H^N$,  list the colors of the edges traversed in a path between $u$ and $v$: $j_1, \ldots, j_s$ (this is equivalent to the set of coordinates in which $u$ and $v$ differ). 
     Then we will call  the vertex ${\rm rot}_u(v)$ obtained by traveling from $u$ along the colors $(j_1+1) \bmod N, \ldots, (j_s+1) \bmod N$, the \emph{rainbow rotation of $v$ from $u$}. For $h \in \Gamma^N$,  the \emph{rainbow rotation from $u$} ${\rm rot}_u\colon \Gamma^N \to \Gamma^N$ is given by ${\rm rot}_u(h)(v) = h({\rm rot}_u(v))$. 

 \end{definition}

\begin{proposition}\label{prop:heightoperations}
    
    Let $b$ be the unique (base) vertex such that $\nu_k(b) = W_k^+$ for each $k = 1, \ldots, 5$, and fix $u \in H^5$. Then, 
    \[
        \nu_k(D_{{\rm rot}_b(h)}) = \nu_{\ell}(D_{h}), \text{ where }k \equiv l+1 \mod 5.
    \] 
    
\end{proposition}
\begin{proof}
    Note that if $v$ is a critical point of $h$, then ${\rm rot}_b^{-1}(v)$ is a critical point of $h' = {\rm rot}_b(h)$. The color of vertex $v$ is determined by the parity of the distance of a path from $b$ to $v$. Note that this parity is the same for the distance from $b$ to ${\rm rot}_b^{-1}(v)$, since the path between $b$ and $v$ can be rotated backwards according to the rainbow to get to ${\rm rot}_b^{-1}(v)$. Thus, we have that 
    \[
    \nu_k({\rm rot}_b^{-1}(v))=
        S_{k,b}({\rm rot}_b^{-1}(v)) \nu_{k}(v). 
    \]
    Recall that $S_{k,b}({\rm rot}_b^{-1}(v))$ is determined by the number of instances of $k$ and $k+2$ in a minimal path between $b$ and ${\rm rot}_b^{-1}(v)$. This is equivalent to the number of instances of $(k+1) \bmod  5$ and $(k+3) \bmod  5$ in a path between $b$ and $v$. Thus, if $\ell \equiv (k + 1) \bmod 5$, we have $\nu_k({\rm rot}_b^{-1}(v))=S_{\ell,b}((v)) \nu_{k}(v)$, and 
    \[
    \nu_k(D_{{\rm rot}_b(h)}) = \nu_{\ell}(D_{h}).
    \]

\end{proof}

Clearly, ${\rm rot}_u: \Gamma^N \to \Gamma^N$ is a bijection. Thus, Proposition~\ref{prop:heightoperations} shows that the distribution of divisors $\nu_k(D_h)$ on $E_k$ will be the same for each $k = 1, \ldots, 5$. 

\begin{theorem}\label{thm:smallerthan8}
  Let $a\cdot \ee{k}{\infty} + b\cdot \ee{k}{4}+c\cdot \ee{k}{2} = \nu_k(D_h)$ for  a height $h$ on $H^5$. Then $0\leq |a| \leq 8$,   $b \equiv -a \mod 4$, and $c = 0$. Moreover, the divisors $\nu_k(D_h)$ occur with the same frequencies displayed in Table~\ref{tab:heightsdivisors_sample}.
  
\end{theorem}

Together, Proposition~\ref{prop:raiselower} and Theorem~\ref{thm:smallerthan8} yield our \hyperref[thm:main]{Main Theorem}.

Note that there are 48 heights $h$ for which $\nu_3(D_h) = \pm 8(\ee{\infty}{3}-\ee{4}{3})$. These come from two  families of heights as shown in Figure \ref{fig:max_heights}. 
The ranked adinkras represented on the left pane of Figure~\ref{fig:max_heights} are those with two pinned vertices $v_1$ and $v_2$, such that, beginning at one pinned vertex $v_1$, traveling along colors $j$ and $(j+2)\bmod 5$ gives pinned vertex $v_2$. For example, one such set of pinned vertices might be $\{ (1,1,1,1,1),  (0,1,1,0,1)\}$. The ranked adinkras represented in the right pane of Figure~\ref{fig:max_heights} are those with 4 pinned vertices $v_1, \ldots, v_4$ with adjacency as follows: beginning at one pinned vertex $v_1$, travel colors $j$ and $j+2$ to arrive at pinned vertex $v_2$, continue along colors $(j-1) \bmod 5$ and $(j+1) \bmod 5$ to pinned vertex $v_3$, and finally along colors $j$ and $(j+2) \bmod 5$ to arrive at pinned vertex $v_4$. For example, one such set of pinned vertices might be $\{(1,1,1,1,1),  (0,1,1,0,1),  (0,1,0,0,0), (1,1,0,1,0)\}$.  If we set the rightmost vertex in the top row of each figure to be the vertex $(1,1,1,1,1)$ and set $j = 4$, then the divisors on ${\rm Jac}(X)$ can be calculated as described in Sections \ref{S:discreteMorse} and \ref{S:Geometrization}. For the height shown in the left pane, the divisor on $E_1 \times \cdots \times E_5$ is 
\begin{equation}\label{eq:extrem1}
\big(O_1,  O_2, -8(\ee{\infty}{3}-\ee{4}{3}),O_4,-8(\ee{\infty}{5}-\ee{4}{5}) 
\big).
\end{equation}
For the height shown in the right pane, the divisor on $E_1 \times \cdots \times E_5$ is 
\begin{equation}\label{eq:extrem2}
\big(O_1, O_2, -8(\ee{\infty}{3}-\ee{4}{3}), -8(\ee{\infty}{4}-\ee{4}{4}),O_5 
\big).
\end{equation}

\begin{figure}
    \includegraphics[width = .48\textwidth]{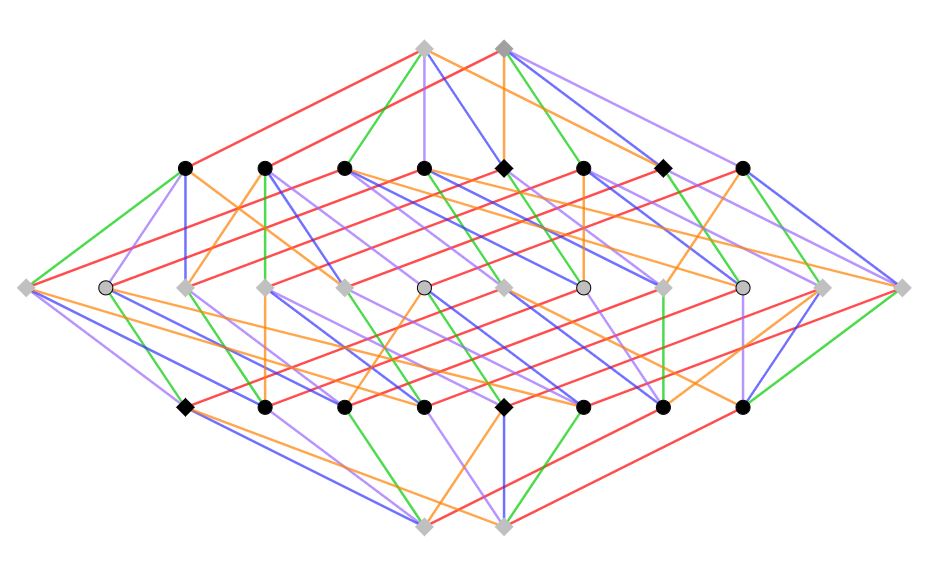}\quad
    \includegraphics[width = .48\textwidth]{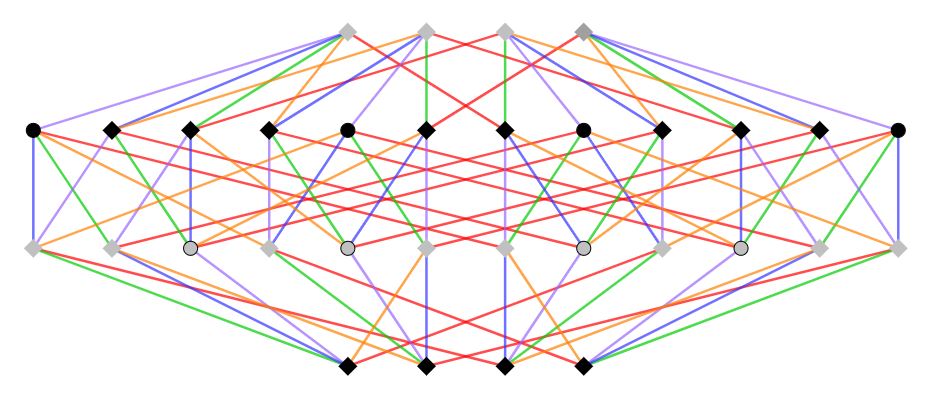}
    \caption{Representatives from the two families of heights for which $\nu_3(D_h) =\pm 8(\ee{\infty}{3}-\ee{4}{3})$. Here, the rainbow ordering for the edges is red-orange-green-blue-purple. The critical point vertices are diamond shaped and the regular vertices are circles.}\label{fig:max_heights}
\end{figure}

Choosing different vertices for the ``pinned'' positions, related in the same way by rainbow ordering (i.e., different representatives from the geometric equivalence classes) will yield similar results, up to signs  and cyclic rotation of the zero and nonzero coordinates. These are the only divisor images for which one of their coordinates has $|a|  = 8$ in the notation of Theorem~\ref{thm:smallerthan8}. 

While the result in Theorem~\ref{thm:smallerthan8} significantly restricts possible divisor images $\eta(D_h)$ on ${\rm Jac}(X)$, there are still 17 possibilities in each coordinate, and thus $17^5 =1\, 419\, 857 $ possible divisor images. This is far more than the total number of heights on $H^5$. However, as demonstrated for the ``extremal'' cases shown in Figure~\ref{fig:max_heights} and calculated in Equations \eqref{eq:extrem1} and \eqref{eq:extrem2},  few among these possibilities actually occur. A natural open question is the following:

\begin{problem}
    Describe the images $\nu(D_h)$ of the divisors $D_h$ for all heights $h$ on $H^5$, and how this description relates to adjacency in the digraph $\Gamma_5$. 
\end{problem}

\subsection*{Acknowledgements}
We thank Steven Charlton for pointing us toward the connection between height functions and 3-colorings of graphs. We also thank the Isaac Newton Institute for Mathematical
Sciences, Cambridge, for support and hospitality during the K-theory, algebraic cycles and
motivic homotopy theory programme, where work on this paper was undertaken. This
work was supported by EPSRC grant no EP/R014604/1.

\end{document}